\documentclass[a4paper,10pt]{article}
\usepackage[utf8]{inputenc}
\usepackage{latexsym,amssymb,amsmath,amscd,amscd,amsthm,amsxtra}
\usepackage[dvips]{graphicx}
\usepackage{xypic}

\newtheorem{theorem}{Theorem}[section]
\newtheorem{lemma}[theorem]{Lemma}
\newtheorem{corollary}[theorem]{Corollary}
\newtheorem{proposition}[theorem]{Proposition}
\newtheorem{example}[theorem]{Example}
\newtheorem{remark}[theorem]{Remark}
\newtheorem{definition}[theorem]{Definition}

\newcommand{\be }{\begin{eqnarray*}}
\newcommand{\ee }{\end{eqnarray*}}

\newcommand{\bE}{\mathbf{E}}
\newcommand{\bF}{\mathbf{F}}
\newcommand{\bH}{\mathbf{H}}
\newcommand{\cA}{\mathcal{A}}
\newcommand{\cB}{\mathcal{B}}

\newcommand{\cI}{\mathcal{I}}

\def\gpd{\,\lower1pt\hbox{$\longrightarrow$}\hskip-.24in\raise2pt
         \hbox{$\longrightarrow$}\,}

\begin{document}
	\title{
	{On  (co-)morphisms  of $n$-Lie-Rinehart algebras with applications to Nambu-Poisson manifolds} }
\author{Yanhui Bi\thanks{{E-mail:}biyanhui0523@163.com}\\
	{\footnotesize
		Center for Mathematical Sciences, College of Mathematics and Information Science,}
	\\
	{\footnotesize Nanchang Hangkong University,		Nanchang 330063, PR China}\\
	Zhixiong Chen\thanks{{E-mail:}chenzhixiong0908@foxmail.com}\\
	{\footnotesize
		College of Mathematics and Information Science, Nanchang Hangkong University,	Nanchang 330063, PR China}  \\
	Tao Zhang\thanks{{E-mail:} zhangtao@htu.edu.cn}\\
	{\footnotesize College of Mathematics, Henan Normal University, Xinxiang 453007, PR China}}	
\date{}
\footnotetext{The research is supported by the National Natural Science Foundation of China (NSFC)  grants 11961049(Bi and Zhang), and 11601219(Bi and Zhang), and by  the Key Project of Jiangxi Natural Science Foundation grant 20232ACB201004(Bi, Chen and  Zhang).}
\maketitle
\begin{abstract}
	In this paper, we give a unified description of morphisms and comorphisms of $n$-Lie-Rinehart  algebras. We show that these morphisms and comorphisms   can be regarded as  two subalgebras of   the $\psi$-sum of $n$-Lie-Rinehart algebras. We also provide similar descriptions for morphisms and comorphisms of $n$-Lie algebroids.
	It is proved that the category of vector bundles with  Nambu-Poisson structures of rank $n$ and the category of their dual bundles with  $n$-Lie algebroid structures of rank $n$ are equivalent to each other.
\end{abstract}
Keywords: $n$-Lie-Rinehart  algebras; Leibniz-Rinehart algebras; morphisms;
comorphisms; Nambu-Poisson structures; $n$-Lie algebroids.
\section{Introduction}
The  $n$-Lie-Rinehart  algebra is a generalized structure of   Lie-Rinehart  algebras. It consists of a quadruple $(\bE,[\cdotp,\cdots,\cdotp],\rho,\cA)$, 
where $\cA$ is a commutative algebra,  $\bE$ is an $\cA$-module and $(\bE,[\cdotp,\cdots,\cdotp])$ is an $n$-Lie algebra,
$\rho: \wedge^{n-1} \bE\to \operatorname{Der}(\cA)$ (called the anchor of $\bE$),
satisfying some compatibility conditions, see Definition  \ref{def:n-lie-renhart}.
The concept of an $n$-Lie algebroid
is a generalization of  Lie algebroids, which carries an $n$-Lie-Rinehart  algebra structure  in the differential geometric  context.
In \cite{ref21}, the author give a category equivalence between  vector bundles with  Poisson structures  and their dual bundles with Lie algebroid structures. In general, there is no one-to-one correspondence between a linear Nambu-Poisson structure  on $E$ and  a Filippov $n$-algebroid structure on $E^*$, see \cite{ref3}. However, it is pointed out in \cite{BC} that there exists an one-to-one correspondence between a linear Nambu-Poisson structure of rank $n$ on $E$ and  a Filippov $n$-algebroid structure of rank $n$ on $E^*$. It is natural to ask whether there is a category equivalence between the category of  vector bundles with  Nambu-Poisson structures of rank $n$  and  the category of their dual bundles with   $n$-Lie algebroid structures of rank $n$.

In order to obtain the equivalent relation between vector bundles with  Nambu-Poisson structures  of rank $n$ and their dual bundles with  $n$-Lie algebroid structures of rank $n$, we define  morphisms and comorphisms of $n$-Lie algebroids. Using  morphisms and comorphisms of $n$-Lie algebroids, we prove the  category equivalence relationship  between them.	Thanks to the morphisms  and comorphisms of Lie-Rinehart  algebras introduced by Z. Chen and Z. J. Liu \cite{ref5}, we obtained the equivalence in this $n$-Lie-Rinehart  algebra setting.

Lie-Rinehart  algebras can be found in \cite{ref5,ref17,ref18}. The $n$-Lie algebra first introduced by Filippov \cite{ref9}. The $n$-Lie algebra is very different form Lie algebras by the multilinear $n$-bracket.
The $n$-Lie algebra \cite{ref2,ref9}, the Nambu-Poisson structure \cite{ref1,ref3,ref8,ref10,ref20,ref22,ref25} and the $n$-Lie algebroid  play important roles in mathematics and physics  \cite{ref10}. The $n$-Lie algebroid  first is defined by
Grabowski and Marmo \cite{ref10}. 	
Recently, Hassine e.t.  studied the representation, cohomology, abelian extension theory of $n$-Lie-Rinehart  algebras in \cite{ref11}.
In particular, in that paper, the homomorphism $\Psi: (\bE,\cA)\to (\bF,\cA)$ of $n$-Lie-Rinehart  algebras is over the same base algebra $\cA$. In this paper, we introduce the concepts of  morphisms  and comorpisms from $(\bE,\cA)$ to $(\bF,\cB)$ over the different base algebra which is not given  in \cite{ref11}.

The first main aim of this paper is to show that   morphisms  and comorpisms of $n$-Lie-Rinehart  algebras can be unified via restriction theory.
The conclusion is that both of them are subalgebras in an $n$-Lie-Rinehart algebra
called the $\psi$-sum of $n$-Lie-Rinehart algebras.
The second main aim of this paper is to show that there is a category equivalence between  vector bundles with  Nambu-Poisson structures of rank $n$ and their dual bundles with  $n$-Lie algebroid structures of rank $n$. We obtain the result that  morphisms (comorphisms) of $n$-Lie-Rinehart algebras can be seen as comorphisms (morphism) of $n$-Lie algebroids (see, Examples (\ref{example5.4}) and (\ref{example5.10})).
As applications, we also study Nambu-Poisson submanifolds using Nambu-Poisson relations and results obtained above.

The paper is organized as follows. In section 2, we recall the notions of $n$-Lie algebras, $n$-Lie algebroids and Nambu-Poisson manifolds.
In section 3, we recall the notions of $n$-Lie-Rinehart  algebras and Leibniz-Rinehart algebras.
We introduce the restricting $n$-Lie-Rinehart  algebras and Leibniz algebras. Applying this process, we obtain the $\psi$-sum of $n$-Lie-Rinehart  algebras and Leibniz-Rinehart algebras.
In section 4, we give two kinds of morphisms of $n$-Lie-Rinehart  algebras. For the comorphism of  $n$-Lie-Rinehart  algebras, we give  Proposition \ref{prop4.6}. It provides a  relationship between comorphisms of $n$-Lie-Rinehart  algebras and differential operators of degree $n-1$ on $\wedge^{n-1}_\cA\bE^*_\cA$.  The principal objective in
this section is the proof of our main result (Theorem \ref{thero4.7}) in this paper. It provides a picture of the
relationship of the two different morphisms, where their graphs turn out to be two subalgebras of the $\psi$-sum
with respect to a given algebra morphism $\psi$.
In section 5, we recall the notion of comorphisms of vector bundles, and then we give the definition of comorphisms of $n$-Lie algebroids. We give a category equivalence between $\mathcal{VB}_{Nambu}$  and $\mathcal{LA^\vee}$, see Theorem \ref{theo5.5}.
In Section 6,  we introduce the notion of Nambu-Poisson submanifolds. We recall the coisotropic submanifold of Nambu-Poisson manifolds. As an application of coisotropic submanifolds, we give the notion of Nambu-Poisson relations. We use the Nambu-Poisson relation to prove the fact that there is a category equivalence between $\mathcal{VB}^\vee_{Nambu}$  and $\mathcal{LA}$, see Theorem \ref{theo6.10}.

\section{Preliminaries}
In this section, we recall  the definitions and some notations of $n$-Lie algebras, Leibniz-Rinehart algebras, $n$-Lie algebroids and Nambu-Poisson manifolds. Let $\mathcal{A}$ is a  commutative associative algebra over $K$, where $K$ is the number
field $\mathbb{R}$ or $\mathbb{C}$.
\begin{definition}[\cite{ref9}]
	An $n$-Lie algebra is a vector space $V$ equipped with an $n$-ary totally skew-symmetric multilinear map (called the $n$-bracket) $[\cdotp,\cdots,\cdotp]: \wedge^{n-1}V \to V$ such that for all $X_1,\cdots,X_{n-1}$, $Y_1,\cdots,Y_n\in V$,
	\begin{equation}
		[X_1,\cdotp \cdotp\cdotp,X_{n-1},[Y_1,\cdotp\cdotp \cdotp,Y_n]]=\sum_{i=1}^{n}[Y_1,\cdots,Y_{i-1},[X_1,\cdotp \cdotp\cdotp,X_{n-1},Y_i],Y_{i+1},\cdots,Y_n].
	\end{equation}
\end{definition}
The elements in $\wedge^{n-1}V$ are called fundamental elements. On $\wedge^{n-1}V$,  these is a new non-symmetric bracket $[\cdotp,\cdotp]_{\wedge^{n-1}V}$  given by
\begin{equation}\label{leibniz bracket}
	[x,y]_{\wedge^{n-1}V}=\sum_{i=1}^{n-1}Y_1\wedge \cdots Y_{i-1}\wedge [X_1,\cdots,X_{n-1},Y_i]\wedge Y_{i+1}\wedge \cdots \wedge Y_{n-1,}	
\end{equation}
for all $x=X_1\wedge \cdots \wedge X_{n-1}$ and $y=Y_1\wedge \cdots \wedge Y_{n-1}$.
This $(\wedge^{n-1}V,[\cdotp,\cdotp]_{\wedge^{n-1}V})$ is a Leibniz algebra, see \cite{ref25}.
\begin{definition}
	A representation of an $n$-Lie algebra $(V,[\cdotp,\cdots,\cdotp])$ on a vector space $W$
	is a multilinear map $\rho: \wedge^{n-1}V \to \operatorname{gl}(W)$ such that for all $X_1,\cdots,X_{n-1},Y_1,\cdots,Y_n\in V$, we have
	\begin{eqnarray}
		\notag[\rho(X_1,\cdots,X_{n-1}),\rho(Y_1,\cdots,Y_{n-1})]		&=&\sum_{i=1}^{n-1}\rho(Y_1,\cdots,Y_{i-1},[X_1,\cdotp \cdotp\cdotp,X_{n-1},Y_i],Y_{i+1},\cdots,Y_{n-1}),\\
		\\
		\notag\rho(X_1,\cdots,X_{n-2},[Y_1,\cdots,Y_n])	&=&\sum_{i=1}^{n}(-1)^{n-i}\rho(Y_1,\cdots,\widehat{Y_i},\cdots,Y_n)\circ \rho(X_1,\cdots,X_{n-2},Y_i).\\
	\end{eqnarray}
\end{definition}
\begin{definition}[\cite{ref11}]
	A Leibniz-Rinehart algebra over $\mathcal{A}$ is a tuple $(E,[\cdotp,\cdotp],\rho,\mathcal{A})$, where $\mathcal{A}$ is
	a  commutative algebra, $E$ is an $\mathcal{A}$-module, $[\cdotp,\cdotp]: E\times E\to
	E$ is a bilinear map
	and the anchor map $\rho: E\to Der(\mathcal{A})$ satisfying the following conditions:
	\begin{enumerate}
		\item[1)] The pair $(E,[\cdotp,\cdotp])$ is a Leibniz algebra;
		\item[2)]$\rho([X,Y])=\rho(X)\rho(Y)-\rho(Y)\rho(X)$;
		\item[3)]$\rho$(aX)=a$\rho(X)$;
		\item[4)] 
		$[X,aY]=[X,Y]+\rho(X)(a)Y,$
	\end{enumerate}
	\noindent for all $a\in \mathcal{A},X,Y\in E$.
\end{definition}
\begin{definition}[\cite{ref10}]
	An $n$-Lie algebroid is a vector bundle $A \to M$ equipped with an
	$n$-bracket $[\cdotp,\cdots,\cdotp]$ on the section spaces  $\Gamma(A)$ of $A$ and a vector bundle map $\rho:\wedge^{n-1}A \to TM$ over	 a manifold $M$, called the anchor of the $n$-Lie algebroid, such that
	\begin{enumerate}
		\item[1)]$(\Gamma(A),[\cdotp,\cdots,\cdotp])$ is an $n$-Lie algebra;
		\item[2)]The anchor map  $\rho:\wedge^{n-1}A \to TM$ satisfies the following relations:
		\begin{enumerate}
			\item[(i)] \begin{eqnarray}
				[\rho(X_1, \cdots, X_{n-1}), \rho(Y_1, \cdots, Y_{n-1})]
				= \sum_{i}\rho(Y_1,\cdots, [X_1,\cdots,X_{n-1},Y_i], \cdots, Y_{n-1})
			\end{eqnarray}
			\item[(ii)]	\begin{eqnarray}
				[X_1,\cdots,X_{n-1},fX_n]=f[X_1,\cdots,X_{n-1},X_n]+\rho(X_1,\cdots, X_{n-1})(f)X_n
			\end{eqnarray}
			for all $X_i,Y_i\in \Gamma(A)$ and $f \in C^\infty(M)$.
		\end{enumerate}
	\end{enumerate}
\end{definition}
Obviously, an $n$-Lie algebroid over a point is an $n$-Lie algebra.
\begin{example}[\cite{ref10}]
	The tangent bundle $T\mathbb{R}^m\to \mathbb{R}^m$ has a structure of $n$-Lie algebroid uniquely determined by
	$$\left[\frac{\partial}{\partial x_{i_1}},\cdots,\frac{\partial}{\partial x_{i_n}}\right]=0$$
	and an anchor map $\rho: \wedge^{n-1}T\mathbb{R}^m\to T\mathbb{R}^m $ is given by $d{x_1}\wedge \cdots \wedge d{x_{n-1}}\otimes \frac{\partial}{\partial x_{1}}$, where $x_1,\cdots,x_n,\cdots,x_m$ are coordinates of $\mathbb{R}^m$.
\end{example}
\begin{remark}
	For $n>2$, the cotangent bundle $T^*M$ is not an $n$-Lie algebroid, since the Jacobi identity is satisfied only for closed 1-forms.
\end{remark}
\begin{definition}[\cite{ref22}]
	A Nambu-Poisson manifold is a smooth manifold $M$ equipped with a Nambu-bracket $\{\cdotp,\cdots,\cdotp\}$ of orede $n$ over $M$ such that it is an $n$-multilinear mapping
	\begin{eqnarray*}
		&\{\cdotp,\cdots,\cdotp\}:C^\infty(M)\times \cdots \times C^\infty(M) \to C^\infty(M)
	\end{eqnarray*}
	and satisfies the following relations:
	\begin{enumerate}
		\item[1.]  Skew-symmetric, $\{f_1,\cdots,f_n\}=sign(\sigma)\{f_{\sigma_1},\cdots,f_{\sigma_n}\},\hspace*{0.5em}\forall \sigma \in \sum_n$;
		\item[2.]Leibniz rule, $\{f_1,\cdots,gf_n\}=\{f_1,\cdots,f_n\}g+\{f_1,\cdots,g\}f_n$;
		\item[3.]Fundamental identity,
		$$\{f_1,\cdots,f_{n-1},\{g_1,\cdots,g_{n}\}\}=\sum_{i=1}^n\{g_1,\cdots,g_{i-1},\{f_1,\cdots,f_{n-1},g_i\},\cdots,g_n\}.$$
	\end{enumerate}
	Here $f_i,g_j,g\in C^\infty(M)$ and $\sum_n$ is a permutation group of $\{1,\cdots,n\}$.
\end{definition}
Note that Poisson manifolds are Nambu-Poisson manifolds of order 2.
\section{The $\psi$-sum of $n$-Lie-Rinehart algebras}
\subsection{$n$-Lie-Rinehart  algebras}
In the paper, we adopt the following notations from  \cite{ref5,ref11}.
\begin{definition} [\cite{ref11}]\label{def:n-lie-renhart}
	$(\mathbf{E},[\cdotp,\cdots,\cdotp],\rho, \cA)$ is called an $n$-Lie-Rinehart algebra  over $\mathcal{A}$, where
	$\mathbf{E}$ be an $\mathcal{A}$-module and an $n$-Lie algebra over $K$ with  a map $\rho:\wedge^{n-1}\mathbf{E}\to \operatorname{Der}(\mathcal{A})$ (called the anchor of $\mathbf{E}$) such that
	\begin{enumerate}
		\item[1.] The map $\rho$ is a representation of $(\mathbf{E},[\cdotp,\cdots,\cdotp])$ on $\cA$.
		\item[2.] The follows conditions are satisfied:
		\begin{eqnarray}
			\label{nlra01}
			{\rho(aX_1,X_2,\cdots,X_{n-1})}&=&a\rho(X_1,X_2,\cdots,X_{n-1}),\\		
			\label{nlra02}	{[X_1,\cdots,X_{n-1},aX_n]}&=&a[X_1,\cdots,X_{n-1},X_n]+\rho(X_1,\cdots, X_{n-1})(a)X_n,
		\end{eqnarray}
	\end{enumerate}
	for all $X_i\in \mathbf{E},a\in \mathcal{A}.$
\end{definition}

From the first equation \eqref{nlra01}, we get
$\rho(X_1,\cdots,aX_i,\cdots,X_{n-1})=(-1)^{i-1}\rho(aX_i,X_1,\cdots,X_{n-1})=(-1)^{i-1}a\rho(X_i,X_1,\cdots,X_{n-1})=a\rho(X_1,\cdots,X_i,\cdots,X_{n-1}),\hspace*{0.5em}\forall 1\leq i \leq n-1$.

We prefer to write
$$\rho(X_1,\cdots, X_{n-1})(a)=[X_1,\cdots,X_{n-1},a]:=(-1)^{n-i}[X_1,\cdots,X_{i-1},a,X_i,X_{i+1},\cdots,X_{n-1},],$$
then   the second Equation \eqref{nlra02} is equivalent to
$${[X_1,\cdots,X_{n-1},aX_n]}=a[X_1,\cdots,X_{n-1},X_n]+[X_1,\cdots,X_{n-1},a]X_n.$$

Let $(\bE,[\cdotp,\cdots,\cdotp],\rho)$ be an $n$-Lie-Rinehart  algebra and denote $\wedge^{n-1}_\cA \bE$ by $\mathcal{E}$. Define a linear map $\widehat{\rho}: \mathcal{E} \to \operatorname{Der}(\cA)$ by
$$\widehat{\rho}(x)=\rho(X_1,\cdots,X_{n-1}),$$
for $x=X_1\wedge \cdots \wedge X_{n-1}\in \bE$.
Obviously, $\mathcal{E}$ is also an $\cA$-module.
\begin{proposition}[\cite{ref11}]\label{prop 3.2}
	With the above notations, $(\mathcal{E},[\cdotp,\cdotp]_\mathcal{E} ,\widehat{\rho})$ is a Leibniz-Rinehart algebra over $\cA$,
	where the bracket is defined as in Equation $\eqref{leibniz bracket}$.
\end{proposition}

Given an $n$-Lie-Rinehart  algebra $(\bE,[\cdotp,\cdots,\cdotp],\rho,\cA)$ and let $\cI\varsubsetneq \cA$ be an ideal of $\cA$. Then, we define
$\bE^\cI\subseteq \bE$ such that
\begin{equation*}
	\bE^\cI=\{X_i, i=1,\cdots, n-1|[X_1,\cdots,X_{n-1},\cI]\subseteq \cI\}, 
\end{equation*}
which is exactly a submodule of $\bE$, and $(\bE^\cI,[\cdotp,\cdots,\cdotp],\cA)$ is an $n$-Lie-Rinehart subalgebra of $\bE$:
\begin{equation*}
	[\bE^\cI,\cdots, \bE^\cI]\subset\bE^\cI,
\end{equation*}

\begin{remark}
	Obviously, when $n=2$, $(\bE^\cI,[\cdotp,\cdotp],\cA)$ is a Lie-Rinehart  algebra.
\end{remark}
We  define $\mathcal{E}^\cI=\{x\in\mathcal{E}|[x,\cI]_\mathcal{E}\subset \cI \}$. Obviously, the subset $\mathcal{E}^\cI$ is a submodule of $\mathcal{E}$, and $(\mathcal{E}^\cI,[\cdotp,\cdotp])$  is a Leibniz-Rinehart subalgebra.

Now, let
$$\cI \bE:=\{\sum_{i}a_iX_i|a_i\in\cI,X_i\in E\}\subset \bE^\cI,$$
then we have $\cI \bE$ is an ideal of $\bE^\cI$.
Using the terminology of quotient module
allows us to extend the results of Z. Chen and Z. J. Liu \cite{ref5} from Lie-Rinehart  algebras to $n$-Lie-Rinehart  algebras.
\begin{lemma}
	The quotient $(\cA/\cI)$-module $\bE^\cI/(\cI \bE)$ is an $n$-Lie-Rinehart  algebra. 
\end{lemma}
\begin{proof}
	By $\bE^\cI$ is a $\cA$-module and $\cA(\cI \bE)\subset \cI \bE$, $\cI \bE^\cI \subset \cI \bE$, we have that $\bE^\cI/(\cI \bE)$ is an $(\cA/\cI)$-module. Now we define the induced bracket by the following equations
	\begin{equation*}
		[\overline{X_1},\cdots,\overline{X_n}]:=\overline{[X_1,\cdots,X_n]},\hspace*{1em}[\overline{X_1},\cdots,\overline{a}]:=\overline{[X_1,\cdots,a]},
	\end{equation*}
	where $\overline{X}=X+\cI \bE$ and $\overline{a}=a+\cI$. Then, we need only to prove that they are well-defined. In fact, we have
	\begin{eqnarray*}
		&[\cI \bE,\cdots,\cI \bE,\cA]\subset \cI,
		\\
		&[\bE^\cI,\cdots,\bE^\cI,\cI]\subset \cI,
		\\
		&[\bE^\cI,\cdots,\bE^\cI,\cI\bE]\subset \cI\bE.
	\end{eqnarray*}
	The first two formulas are obvious. For the last formula, notice that if $X_1,\cdots,X_{n-1}\in \bE^\cI,aX_n\in \cI \bE$, where $a\in \cI,X_n\in \bE$, we have
	$$[X_1,\cdots,X_{n-1},aX_n]=a[X_1,\cdots,X_{n-1},X_n]+[X_1,\cdots,X_{n-1},a]X_n,$$
	where the two terms in the right hand side of the above equation is in  $\cI\bE$ since $[X_1,\cdots,X_{n-1},a]\in   \cI$.
	Therefore, we  obtain that $(\bE^\cI/(\cI \bE),\cA/\cI)$ inherits all structures of $(\bE,\cA)$.
\end{proof}
Let $\bE$ be an $\cA$-module and $\cI$ be an ideal of $\cA$. Then, $\bE/(\cI \bE)\cong \bE \otimes_\cA(\cA/\cI)$ as $\cA/\cI$-module, under the isomorphism $\sigma: \overline{X} \to X\otimes_\cA \overline{1}$. For the proof, see \cite{ref5}.

By the above results, we have the following Lemma.
\begin{lemma}\label{lemma1.4}
	Given an $n$-Lie-Rinehart algebra  $(\bE,\cA)$ and an ideal $\cI\subset \cA$. Then under the isomorphism $\sigma$ defined above, we have that $\bE^\cI/(\cI \bE)\cong \bE^\cI \otimes_\cA(\cA/\cI)$ and the latter has the induced $n$-Lie-Rinehart  algebra structure (over $\cA/\cI$) defined by
	\begin{eqnarray*}
		&[X_1\otimes_\cA \overline{a_1},\cdots,X_{n-1}\otimes_\cA \overline{a_{n-1}},\overline{a_n}]=\overline{a_1\cdots a_{n-1}}\overline{[X_1,\cdots,X_{n-1},a_n]}
		\\
		&[X_1\otimes_\cA \overline{a_1},\cdots,X_{n-1}\otimes_\cA \overline{a_{n-1}},X_n\otimes_\cA \overline{a_n}]=[X_1,\cdots,X_n]\otimes_\cA \overline{a_1\cdots a_n}
		\\
		&+\sum_{i=1}^{n}(-1)^{n-i}X_i\otimes_\cA \overline{a_1\cdots \widehat{a_i}\cdots a_n}\overline{[X_1,\cdots \widehat{X_i},\cdots,X_n,a_i]},
	\end{eqnarray*}
	for all $X_i\in \bE^\cI,a_i\in \cA$.
\end{lemma}
\begin{proof}
	By the isomorphism map $\sigma$, we have that the image of $\bE^\cI/(\cI \bE)$ is exactly is $\bE^\cI \otimes_\cA (\cA/\cI)$. Therefore, we define the bracket in $\bE^\cI \otimes_\cA (\cA/\cI)$,simply given by
	\begin{eqnarray*}
		&[X_1\otimes_\cA \overline{1},\cdots,X_{n-1}\otimes_\cA \overline{1},\overline{a_n}]=\overline{[X_1,\cdots,X_{n-1},a_n]}
		\\
		&[X_1\otimes_\cA \overline{1},\cdots,X_{n-1}\otimes_\cA \overline{1},X_n\otimes_\cA \overline{1}]=[X_1,\cdots,X_n]\otimes_\cA \overline{1}.
	\end{eqnarray*}
\end{proof}
\begin{definition}\label{de3.6}
	We denote
	\begin{equation*}
		\bE_\cI=\bE^\cI\otimes_\cA(\cA/\cI)=\bE^\cI/(\cI \bE)
	\end{equation*}
	and call $(\bE_\cI,[\cdotp,\cdots,\cdotp]_{\bE_\cI},\cA/\cI)$ the $\cI$-restriction of an $n$-Lie-Rinehart  algebra $(\bE,[\cdotp,\cdots,\cdotp]_\bE,\cA)$ with respect to the ideal $\cI \subset \cA$.
\end{definition}
\subsection{The $\psi$-sum}
Similar as the fact that two Lie-Rinehart  algebras $(\bE,[\cdotp,\cdotp]_E,\cA)$ and $(\bF,[\cdotp,\cdotp]_F,\cB)$ define their direct sum to be an $\cA \otimes \cB$-module $(\bE \otimes \cB)\oplus (\cA \otimes \bF)$, we also define direct sum of two  $n$-Lie-Rinehart algebras $(\bE,[\cdotp,\cdots,\cdotp]_E,\cA)$ and $(\bF,[\cdotp,\cdots,\cdotp]_F,\cB)$   such that it is an $\cA \otimes \cB$-module $(\bE \otimes \cB)\oplus (\cA \otimes \bF)$. Then, we have the direct sum with an $n$-Lie-Rinehart algebra  structure.
\begin{proposition}
	$(\bE \otimes \cB)\oplus (\cA \otimes \bF)$ is an $n$-Lie-Rinehart  algebra, where the bracket and anchor map $\rho$ are given by the following rules:
	\begin{eqnarray*}
		&&\rho((X_1\otimes b_1 + a_1\otimes Y_1)\wedge\cdotp\cdotp\cdotp\wedge (X_{n-1}\otimes b_{n-1} + a_{n-1}\otimes Y_{n-1}))(a_n\otimes b_n)		\\
		&=&[X_1\otimes b_1 + a_1\otimes Y_1,\cdots,X_{n-1}\otimes b_{n-1} + a_{n-1}\otimes Y_{n-1},a_n\otimes b_n]\\
		&=&[X_1,\cdots,X_{n-1},a_n]_\bE\otimes b_1\cdots b_n+a_1\cdots a_n\otimes[Y_1,\cdots,Y_{n-1},b_n]_\bF;\\[1em]
		&&[X_1\otimes b_1 + a_1\otimes Y_1,\cdots,X_n\otimes b_n + a_n\otimes Y_n]		\\
		&=&[X_1,\cdots,X_n]_\bE\otimes b_1\cdots b_n+a_1\cdots a_n\otimes[Y_1,\cdots,Y_n]_\bF
		\\
		&&+\sum_{i=1}^{n}(-1)^{n-i}[X_1,\cdots \widehat{X_i},\cdots,X_n,a_i]_\bE\otimes b_1\cdots \widehat{b_i}\cdots b_nY_i
		\\
		&&+\sum_{i=1}^{n}(-1)^{n-i}a_1\cdots \widehat{a_i}\cdots a_nX_i\otimes[Y_1,\cdots \widehat{Y_i},\cdots,Y_n,b_i]_\bF,
	\end{eqnarray*}
	where $a_i\in \cA,b_i\in \cB, X_i\in \bE, Y_i\in \bF$.
\end{proposition}
The proof of the above Proposition  is by direct computations, so we omit the detail.

\begin{remark}
	All the lemmas  in subsection 3.1  hold for Leibniz-Rinehart algebras $(\mathcal{E},[\cdotp,\cdotp]_\mathcal{E},\cA)$. Therefore, given two Leibniz-Rinehart algebras $(\mathcal{E}=\wedge^{n-1}\bE,[\cdotp,\cdotp]_\mathcal{E},\widehat{\rho_\mathcal{E}})$ and $(\mathcal{F}=\wedge^{n-1}\bF,[\cdotp,\cdotp]_\mathcal{F},\widehat{\rho_\mathcal{F}})$, we  define their direct sum to be the $\cA \otimes \cB$-module $(\mathcal{E} \otimes \cB)\oplus (\cA \otimes \mathcal{F})$. Then, we have the direct sum with a new Leibniz-Rinehart algebra structure.
\end{remark}
\begin{proposition}
	$(\mathcal{E} \otimes \cB)\oplus (\cA \otimes \mathcal{F})$ is a Leibniz-Rinehart algebra, where the bracket  and anchor map $\rho$ are given by the following rules:
	\begin{eqnarray*}
		&&[x_1\otimes b_1 + a_1\otimes y_1,a_2\otimes b_2]\\
		&=&[x_1,a_2]_\mathcal{E}\otimes b_1b_2+a_1a_2\otimes[y_1,b_2]_\mathcal{F};\\[1em]
		&&[x_1\otimes b_1 + a_1\otimes y_1,x_2\otimes b_2 + a_2\otimes y_2]
		\\
		&=&[x_1,x_2]_\mathcal{E}\otimes b_1b_2+a_1a_2\otimes[y_1,y_2]_\mathcal{F}
		\\
		&&+[x_1,a_2]_\mathcal{E}\otimes b_1y_2+a_1x_2\otimes[y_1,b_2]_\mathcal{F}\\
		&&+ (-1)^{n-1}a_2x_1\otimes[y_2,b_1]_\mathcal{F}+		(-1)^{n-1}[x_2,a_1]_\mathcal{E}\otimes b_2y_1,
	\end{eqnarray*}
	for all $a_i\in \cA,b_i\in \cB, x_i\in \mathcal{E}, y_i\in \mathcal{F}$.
\end{proposition}
There is a known fact that $\cB$-modules can be regarded as $\cA$-modules via an algebraic homomorphism $\psi: \cA \to \cB$. Then we have the following map
$$\widetilde{\psi}:\cA \otimes \cB \to \cB, \hspace*{1em}a\otimes b\mapsto \psi(a)b.$$
Assume that $\psi$ is surjective and $\cB\cong (\cA \otimes \cB)/\ker\widetilde{\psi}$
\begin{definition}
	Let $\bH=(\bE \otimes \cB)\oplus (\cA \otimes \bF)$ be a direct sum of two $n$-Lie-Rinehart algebras $(\bE,[\cdotp,\cdots,\cdotp]_E,\cA)$ and $(\bF,[\cdotp,\cdots,\cdotp]_F,\cB)$. Let $\cI=\ker\widetilde{\psi}$  be an ideal. We denote $(\bH_\cI,[\cdotp,\cdots,\cdotp]_{\bH_\cI},(\cA \otimes \cB)/\ker\widetilde{\psi}\cong \cB)$ by $(\bE\oplus_\psi \bF,[\cdotp,\cdots,\cdotp]_{\bE\oplus_\psi \bF},\cB)$ or simply $\bE\oplus_\psi \bF$ and  call it the $\psi$-sum of $(\bE,[\cdotp,\cdots,\cdotp]_\bE,\cA)$
	and $(\bF,[\cdotp,\cdots,\cdotp]_\bF,\cB)$ with respect to the morphism $\psi$.
\end{definition}
The $\psi$-sum $\bE\oplus_\psi \bF$ is a $\cB$-submodule of $(\bE \otimes_\cA \cB)\oplus \bF$. For the proof one can see \cite{ref5}.
\begin{remark}
	For the Leibniz-Rinehart algebra $(\mathcal{E} \otimes \cB)\oplus (\cA \otimes \mathcal{F})$, we denote  $(\mathcal{H}_\cI,[\cdotp,\cdot]_{\mathcal{H}_\cI},\widehat{\rho_{\mathcal{E} \oplus_\psi \mathcal{F}}},\cA \otimes \cB/\ker\widetilde{\psi}\cong \cB)$ by
	$\mathcal{E} \oplus_\psi \mathcal{F}$ and call it the $\psi$-sum of $(\mathcal{E},[\cdotp,\cdotp]_\mathcal{E},\widehat{\rho_\mathcal{E}},\mathcal{A})$ and $(\mathcal{F},[\cdotp,\cdotp]_\mathcal{F},\widehat{\rho_\mathcal{F}},\cB)$ with respect to the morphism $\psi$, where $\mathcal{H}=(\mathcal{E} \otimes \cB)\oplus (\cA \otimes \mathcal{F})$.
\end{remark}
By convention, we denote $\mathfrak{X}_j:=X_{(i_j)}\otimes b_{(i_j)}:=\sum_{i_j}X_{i_j}\otimes_\cA b_{i_j}\in \bE \otimes_\cA \cB,\hspace*{0.5em}\text{and}\hspace*{0.5em}\mathfrak{Y}_j:=a_{(i_j)}\otimes Y_{(i_j)} :=\sum_{i_j}a_{i_j}\otimes Y_{i_j}\in \cA\otimes\bF,\hspace*{0.5em}\forall j=1,\cdots, n$.
\begin{theorem}\label{theorem 1.7}
	An element $\mathfrak{X}_1+Y_1\in (\bE \otimes_\cA \cB)\oplus \bF$ belongs to $\bE\oplus_\psi \bF$ if and only if
	\begin{eqnarray}\label{eq1.1}
		\psi([X_{(i_1)},\cdots,X_{(i_{n-1})},a])b_{(i_1)}\cdots b_{(i_{n-1})}=[Y_1,\cdots,Y_{n-1},\psi(a)],
	\end{eqnarray}
	for all $a\in \cA$ and $\mathfrak{X}_2+\mathfrak{Y}_2,\cdots,\mathfrak{X}_{n-1}+\mathfrak{Y}_{n-1} \in H^\cI$.
\end{theorem}
\begin{proof}
	If $\mathfrak{X}_1+\mathfrak{Y}_1\in \bH^\cI$, for all $\mathfrak{X}_2+\mathfrak{Y}_2,\cdots,\mathfrak{X}_{n-1}+\mathfrak{Y}_{n-1}\in \bH^\cI$, we have $\widetilde{\psi}[\mathfrak{X}_1+\mathfrak{Y}_1,\cdots,\mathfrak{X}_{n-1}+\mathfrak{Y}_{n-1},\cI]=0$. An element $\mathfrak{X}_1+\mathfrak{Y}_1\in \bH^\cI$ if and only if
	\begin{eqnarray*}
		&&[\mathfrak{X}_1+\mathfrak{Y}_1,\cdots,\mathfrak{X}_{n-1}+\mathfrak{Y}_{n-1},a\otimes b-1\otimes \psi(a)b]\\
		&=&\psi([X_{(i_1)},\cdots,X_{(i_{n-1})},a]) b_{(i_1)}\cdots b_{(i_{n-1})}b
		+\psi(a_{(i_1)}\cdots a_{(i_{n-1})}a)[Y_{(i_1)},\cdots,Y_{(i_{n-1})},b]
		\\
		&&-\psi(a_{(i_1)}\cdots a_{(i_{n-1})}1)[Y_{(i_1)},\cdots,Y_{(i_{n-1})},\psi(a)b]
		\\
		&=&\psi([X_{(i_1)},\cdots,X_{(i_{n-1})},a]) b_{(i_1)}\cdots b_{(i_{n-1})}b
		-\psi(a_{(i_1)}\cdots a_{(i_{n-1})}1)[Y_{(i_1)},\cdots,Y_{(i_{n-1})},\psi(a)]b
		\\
		&=&\psi([X_{(i_1)},\cdots,X_{(i_{n-1})},a])b_{(i_1)}\cdots b_{(i_{n-1})}-[\psi(a_{(i_1)})Y_{(i_1)},\cdots,\psi(a_{(i_{n-1})})Y_{(i_{n-1})},\psi(a)])b\\
		&=&0
	\end{eqnarray*}
	holds for all $a\in \cA,b\in\cB$. The above equation shows that
	$$\psi([X_{(i_1)},\cdots,X_{(i_{n-1})},a]) b_{(i_1)}\cdots b_{(i_{n-1})}-[\psi(a_{(i_1)})Y_{(i_1)},\cdots,\psi(a_{(i_{n-1})})Y_{(i_{n-1})},\psi(a)]=0,$$
	for all $a\in \cA$ and $\mathfrak{X}_1+\mathfrak{Y}_1,\cdots,\mathfrak{X}_{n-1}+\mathfrak{Y}_{n-1}\in \bH^\cI$.
	
	On the other hand, by the definition of $\psi$-sum, we have
	\[
	\bE\oplus_\psi \bF=\bH^\cI\otimes_{\cA\otimes\cB}\cB\subset\bH\otimes_{\cA\otimes\cB}\cB\cong (\bE\otimes_\cA\cB)\oplus \bF.
	\]	
	Moreover, by Lemma \ref{lemma1.4} and Definition \ref{de3.6}, an element in $\bE\oplus_\psi \bF$ can be written in the following form:
	\[
	X_{(i_j)}\otimes b_{(i_j)}+\psi(a_{(i_j)})Y_{(i_j)}
	\]
	where $X_{(i_j)}\otimes b_{(i_j)}+a_{(i_j)}\otimes Y_{(i_j)}\in \bH^\cI$.
	The proof is finished.
\end{proof}
By the above Lemma \ref{lemma1.4}, it is easy to see  that the expressions of $n$-bracket of $\bE \oplus_\psi \bF$ can be given in the following proposition.
\begin{proposition}\label{prop 1.8}
	The structure maps of the $n$-Lie-Rinehart algebra  $(\bE\oplus_\psi \bF,[\cdotp,\cdots,\cdotp],\bE\oplus_\psi \bF)$ are given by
	\begin{eqnarray*}
		&&[\mathfrak{X}_1+Y_1,\cdots,\mathfrak{X}_{n-1}+Y_{n-1},b]\\
		&=&[Y_1,\cdots,Y_{n-1},b]_\bF;\\[1em]
		&&[\mathfrak{X}_1+Y_1,\cdots,\mathfrak{X}_{n}+Y_{n}]\\
		&=&[X_{(i_1)},\cdots,X_{(i_n)}]_\bE\otimes_\cA b_{(i_1)}\cdots b_{(i_n)}+[Y_1,\cdots,Y_n]_\bF		\\
		&&+\sum_{j=1}^{n}(-1)^{n+j}X_{(i_j)}\otimes_\cA[Y_1,\cdots,\widehat{Y_i},\cdots,b_{(i_j)}]_\bF,
	\end{eqnarray*}
	for all $\mathfrak{X}_1+Y_1,\cdots,\mathfrak{X}_n+Y_{n}\in \bE \oplus_\psi \bF,b\in \cB$.
\end{proposition}
For Leibniz-Rinehart algebras, we have the following proposition:
\begin{proposition}
	The structure maps of the Leibniz-Rinehart algebra  $(\mathcal{E}\oplus_\psi \mathcal{F},[\cdotp,\cdotp],\cB)$ are given by
	\begin{eqnarray*}
		&&[\sum_{i}x_{i}\otimes_\cA b_{i}+y_1,b]=[y_1,b]_{\mathcal{F}};		\\
		&&[\sum_{i}x_{i}\otimes_\cA b_{i}+y_1,\sum_{j}x_{j}\otimes_\cA b_{j}+y_2]=\sum_{i} \sum_{j}[x_{i},,x_{j}]_\mathcal{E}\otimes_\cA b_{i} b_{j}+[y_1,y_2]_\mathcal{F}		\\
		&&+\sum_jx_{j}\otimes_\cA[y_1,b_j]_\mathcal{F}+(-1)^{n-1}\sum_ix_{i}\otimes_\cA[y_2,b_i]_\mathcal{E},
	\end{eqnarray*}
	for all $\sum_{i}x_{i}\otimes_\cA b_{i}+y_1,\cdots,\sum_{j}x_{j}\otimes_\cA b_{j}+y_{2}\in \mathcal{E} \oplus_\psi \mathcal{F},b\in \cB$.
\end{proposition}
Let $(E_1,M_1,\rho_1,[\cdotp,\cdots,\cdotp]_1)$ be an $n$-Lie algebroid with the anchor map $\rho_1:\wedge^{n-1}E_1\to TM_1$.
We will denote the tangent vector $\rho
_1(X_1\wedge\cdots \wedge X_{n-1})$, $X_i\in E_1$ by $[X_1,\cdots,X_{n-1},\cdotp]$.
\begin{example}
	Given an $n$-Lie algebroid $(E_1,M_1,\rho_1,[\cdotp,\cdots,\cdotp]_1)$ and an embedded submanifold $M_2$. Let $i: M_2\to M_1$ be the inclusion. Consider the algebra $\cA=C^\infty(M_1)$ and its ideal $\cI=\ker i^*=\{f\in \cA|f|_{M_2}=0\}$. Denoting $\bE=\Gamma(E_1)$, we  obtain an $n$-Lie-Rinehart algebra $(\bE,[\cdotp,\cdots,\cdotp]_\bE,\rho_1)$ over $\cA$. Then, we want to find the restriction of the $n$-Lie-Rinehart algebra  $(\bE,[\cdotp,\cdots,\cdotp]_\bE,\rho_1)$ with respect to $\cI$. Obviously,
	for all $\sigma_1,\cdots,\sigma_{n-1}\in \Gamma(E_1)\cap\bE^\cI$, we have
	$$\rho_1(\sigma_1\wedge \cdots \wedge \sigma_{n-1})|_{M_2}\in TM_2.$$
	By using the definition of $\cI \bE$ and $\cI$, we have
	$$\cI \bE=\{\sigma\in\Gamma(E_1)|\sigma|_{M_2}=0\}.$$
	Therefore, the quotient algebra $\bE_\cI$ can be regarded as a subspace of the spaces of sections $E_{1,M_2}:=\{e\in E_{1,p}|p\in M_2\}$ such that
	$$\rho_1(e_1\wedge \cdots \wedge e_{n-1})\in T_pM_2,\hspace*{0.5em}\forall e_1,\cdots,e_{n-1}\in \bE_\cI$$
	as an $\cA/\cI\cong C^\infty(M_2)$-module.
\end{example}
\begin{example}
	Given two $n$-Lie algebras $\bE$ and $\bF$ over $K$. The only $K$-algebraic homomorphism from $K$ to $K$ is the identity map $id$, and the $id$-sum of $(\bE,[\cdotp,\cdots,\cdotp]_\bE,K)$ and $(\bF,[\cdotp,\cdots,\cdotp]_\bF,K)$ is exactly $(\bE\oplus \bF,[\cdotp,\cdots,\cdotp]_{\bE\oplus \bF},K)$.
\end{example}
\begin{example}
	Given an $n$-Lie-Rinehart algebra  $(\bE,[\cdotp,\cdots,\cdotp]_\bE,\rho_\bE)$ over $\cA$ and an $n$-Lie algebra $\bF$ over $K$. Let $\psi:\cA \to K$ is an algebraic homomorphism. The $\psi$-sum of $\bE$ and $\bF$ is the set $\bE\oplus_\psi \bF=\{X\otimes_\cA 1+Y|X \in E,Y \in F\}$ such that
	$$[X_1,\cdots,X_{n-1},a]\in \ker\psi, \hspace*{0.5em}\forall X_1\otimes_\cA 1+Y_1,\cdots,X_{n-1}\otimes_\cA 1+Y_{n-1}\in \bE\oplus_\psi \bF,a\in \cA.$$
\end{example}
\begin{example}
	Let $(E_1,M,\rho_1,[\cdotp,\cdots,\cdotp]_1)$ and $(E_2,N,\rho_2,[\cdotp,\cdots,\cdotp]_2)$ be two $n$-Lie algebroids. The direct sum is a bundle $E_1\times E_2$ over $M\times N$, with a bundle map $(\sigma_p,\tau_q)\to (p,q)$, where $(\sigma_p,\tau_q),\sigma_p\in E_{1,p},\tau_q\in E_{2,q}$. We want to find the $n$-Lie algebroid structure of the sum direct $E_1\times E_2$. The anchor map is $$((\sigma_1)_p\wedge \cdots \wedge (\sigma_{n-1})_p,(\tau_1)_q\wedge \cdots \wedge (\tau_{n-1})_q) \to (\rho_1((\sigma_1)_p\wedge \cdots \wedge (\sigma_{n-1})_p),\rho_2((\tau_1)_q\wedge \cdots \wedge (\tau_{n-1})_q)).$$
	For some sections on $E_1\times E_2$, we define the $n$-Lie bracket to be
	\begin{eqnarray*}
		&&[(f_1\sigma_1,g_1\tau_1),\cdots,(f_n\sigma_n,g_n\tau_n)]\\
		&=&(f_1\cdots f_n[\sigma_1,\cdots,\sigma_n]_1+\sum_{i=1}^{n}(-1)^{n-i}g_1\cdots \widehat{g_i}\cdots g_n[\tau_1,\cdots,\widehat{\tau_i},\cdots \tau_n,f_i]_1\sigma_i,		\\
		&&g_1\cdots g_n[\tau_1,\cdots,\tau_n]_2+\sum_{i=1}^{n}(-1)^{n-i}f_1\cdots \widehat{f_i}\cdots f_n[\sigma_1,\cdots,\widehat{\sigma_i},\cdots \sigma_n,f_i]_2\tau_i).
	\end{eqnarray*}
	Here $f_i\in C^\infty(N),g_i\in C^\infty(M),\sigma_i\in \Gamma(E_1),\tau_i\in \Gamma(E_2)$. We regard $C^\infty(M\times N)\cong C^\infty(M)\otimes C^\infty(N)$ and $\Gamma(E_1\times E_2)$ as the $C^\infty(M\times N)$-module $\Gamma(E_1)\otimes C^\infty(N)\oplus \Gamma(E_2)\otimes C^\infty(M)$.
\end{example}
Consider a smooth map $\phi:M \to N$, one has its graph
\begin{equation}
	Gr(\phi)=\{(p,\phi(p))|p\in M\}\subset M\times N.
\end{equation}
Therefore, the $\phi$-sum of $E_1$ and $E_2$ is the set
\begin{equation}
	E_1\oplus_\phi E_2=\{((\sigma)_p,(\tau)_{\phi(p)})\in E_1\oplus \phi^!E_2|\forall p\in M,(\sigma)_p\in(E_1)_p,(\tau)_{\phi(p)}\in (E_2)_{\phi(p)}\}
\end{equation}
such that
\begin{equation}
	\phi_*\circ \rho_1((\sigma_1)_p\wedge \cdots (\sigma_{n-1})_p)=\rho_2((\tau_1)_{\phi(p)}\wedge \cdots (\tau_{n-1})_{\phi(p)}),
\end{equation}
for all $((\sigma_i)_p,(\tau_i)_{\phi(p)})\in E_1\oplus_\phi E_2,p\in M.$
\section{Morphisms and Comorphisms of $n$-Lie-Rinehart algebras }
In this section, we introduce the concepts of morphisms and comorphisms of $n$-Lie-Rinehart  algebras.
The first one is a generalization of the homomorphism of $n$-Lie-Rinehart  algebras defined in \cite{ref11}.
\begin{definition}\label{defini 2.1}
	Given two $n$-Lie-Rinehart  algebras  $(\bE,[\cdotp,\cdots,\cdotp]_\bE,\cA)$ and $(\bF,[\cdotp,\cdots,\cdotp]_\bF,\cB)$. A morphism of $n$-Lie-Rinehart  algebras from $\bE$ to $\bF$ is a pair  $(\bE,[\cdotp,\cdots,\cdotp]_\bE,\cA)\stackrel{{(\Psi,\psi)}}\rightrightarrows (\bF,[\cdotp,\cdots,\cdotp]_\bF,\cB)$ such that
	\begin{enumerate}
		\item[1)]
		\begin{equation}\label{eq2.1}
			\psi([X_1,\cdots,X_{n-1},a]_\bE)=[\Psi(X_1),\cdots,\Psi(X_{n-1}),\psi(a)]_\bF,\hspace*{0.5em}\forall X_1,\cdots,X_{n-1}\in \bE,a\in \cA;
		\end{equation}
		\item[2)]
		\begin{equation}\label{morphism defin 2}
			\Psi([X_1,\cdots,X_{n}]_\bE)=[\Psi(X_1),\cdots,\Psi(X_{n})]_\bF,\hspace*{0.5em}\forall X_1,\cdots,X_{n}\in \bE.
		\end{equation}
	\end{enumerate}
	Here $\psi:\cA \to \cB$ is an algebraic homomorphism and $\Psi:\bE \to \bF$ is a map of $\cA$-modules (considering $\cB$-module as $\cA$-modules through $\psi$). In particular, if both $\Psi$ and $\psi$ are injective, we call  $(\bE,[\cdotp,\cdots,\cdotp]_\bE,\cA)$ a subalgebra of $(\bF,[\cdotp,\cdots,\cdotp]_\bF,\cB)$.
\end{definition}
\begin{definition}\label{equivalent morphism}
	Given two $n$-Lie-Rinehart algebras  $(\bE,[\cdotp,\cdots,\cdotp]_\bE,\cA)$ and $(\bF,[\cdotp,\cdots,\cdotp]_\bF,\cB)$. A comorphism of $n$-Lie-Rinehart algebras  from $\bF$ to $\bE$ is a pair  $(\bF,[\cdotp,\cdots,\cdotp]_\bF,\cB)\stackrel{{(\Psi,\psi)}}\rightleftarrows (\bE,[\cdotp,\cdots,\cdotp]_\bE,\cA)$ such that
	\begin{enumerate}
		\item[1)]
		\begin{equation}\label{comorphism equiva}
			[Y_1,\cdots,Y_{n-1},\psi(a)]_\bF=\sum_{k_1}\cdots \sum_{k_{n-1}}b_{k_1}\cdots b_{k_{n-1}}\psi[X_{k_1},\cdots,X_{k_{n-1}},a]_\bE,\hspace*{0.5em}\forall a\in \cA,
		\end{equation}
		\item[2)]
		\begin{eqnarray}\label{comorphism equv2}
			&\Psi([Y_1,\cdots,Y_{n}]_\bF)=[\Psi(Y_1),\cdots,\Psi(Y_n)]_\bE,
		\end{eqnarray}	
		for all $Y_1,\cdots,Y_{n}\in \bF$ and $\Psi(Y_i)=\sum_{k_i}X_{k_i}\otimes_\cA b_{k_i}$, where $1\leq i \leq n,X_{k_i}\in \bE,b_{k_i}\in\cB$.
	\end{enumerate}
	Here $\psi:\cA \to \cB$ is an algebraic homomorphism and $\Psi:\bF \to \bE\otimes_{\cA}\cB$ is a map of $\cB$-modules. In particular, if  $\Psi$ is injective and $\psi$ is surjective, we call  $(\bF,[\cdotp,\cdots,\cdotp]_\bF,\cB)$ a co-subalgebra of $(\bE,[\cdotp,\cdots,\cdotp]_\bE,\cA)$.
\end{definition}
For Equation \eqref{comorphism equv2}, we obtain the local expression
\begin{eqnarray*}
	&\Psi([Y_1,\cdots,Y_{n}]_\bF)=
	\sum_{k_1}\cdots \sum_{k_n}[X_{k_1},\cdots,X_{k_n}]_\bE \otimes_\cA b_{k_1}\cdots b_{k_n}\notag
	\\
	&\hspace*{13em}+\sum_{i=1}^{n}\sum_{k_i}(-1)^{n-i}X_{k_i}\otimes_\cA [Y_1,\cdots,\widehat{Y_i},\cdots,Y_{n},b_{k_i}]_\bF.
\end{eqnarray*}
The following definitions illustrate morphism and comorphism of Leibniz-Rinehart algebras of this form $(\mathcal{F},[\cdotp,\cdotp]_\mathcal{F},\cB)$.
\begin{definition}
	Given two Leibniz-Rinehart  algebras  $(\mathcal{E},[\cdotp,\cdotp]_\mathcal{E},\cA)$ and $(\mathcal{F},[\cdotp,\cdotp]_\mathcal{F},\cB)$. A morphism of Leibniz-Rinehart  algebras from $\mathcal{E}$ to $\mathcal{F}$ is a pair  $(\mathcal{E},[\cdotp,\cdotp]_\mathcal{E},\cA)\stackrel{{(\Psi,\psi)}}\rightrightarrows (\mathcal{F},[\cdotp,\cdotp]_\mathcal{F},\cB)$ such that
	\begin{enumerate}
		\item[1)]
		\begin{equation}
			\psi([x,a]_\mathcal{E})=[\Psi(x),\psi(a)]_\mathcal{F},\hspace*{0.5em}\forall x\in \mathcal{E},a\in \cA;
		\end{equation}
		\item[2)]
		\begin{equation}
			\Psi([x_1,x_2]_{\mathcal{E}})=[\Psi(x_1),\Psi(x_2)]_\mathcal{F},\hspace*{0.5em}\forall x_1,x_2\in \mathcal{E}.
		\end{equation}
	\end{enumerate}
	Here $\psi:\cA \to \cB$ is an algebraic homomorphism and $\Psi:\mathcal{E} \to \mathcal{F}$ is a map of $\cA$-modules (considering $\cB$-module as $\cA$-modules through $\psi$).
\end{definition}
\begin{definition}
	Given two Leibniz-Rinehart algebras  $(\mathcal{E},[\cdotp,\cdotp]_\mathcal{E},\cA)$ and $(\mathcal{F},[\cdotp,\cdotp]_\mathcal{F},\cB)$. A comorphism of Leibniz-Rinehart algebras  from $\mathcal{F}$ to $\mathcal{E}$ is a pair  $(\mathcal{F},[\cdotp,\cdotp]_\mathcal{F},\cB)\stackrel{{(\Psi,\psi)}}\rightleftarrows (\mathcal{E},[\cdotp,\cdotp]_\mathcal{E},\cA)$ such that
	\begin{enumerate}
		\item[1)]
		\begin{equation}
			[y,\psi(a)]_\mathcal{F}=\sum_{i}b_i\psi[x_i,a]_\mathcal{E},\hspace*{0.5em}\forall a\in \cA,
		\end{equation}
		\item[2)]
		\begin{eqnarray}
			&\Psi([y_1,y_2]_\mathcal{F})=[\Psi(y_1),\Psi(y_2)]_\mathcal{E},
		\end{eqnarray}	
		for all $y_1,y_{2}\in\mathcal{F}$ and $\Psi(y)=\sum_{i}x_{i}\otimes_\cA b_{i}$, where $1\leq i \leq n,x_{i}\in \mathcal{E},b_{i}\in\cB$.
	\end{enumerate}
	Here $\psi:\cA \to \cB$ is an algebraic homomorphism and $\Psi:\mathcal{F} \to \mathcal{E}\otimes_{\cA}\cB$ is a map of $\cB$-modules.
\end{definition}
\begin{remark}\label{remark4.3}
	Given a comorphism of $n$-Lie-Rinehart algebras   $(\bF,[\cdotp,\cdots,\cdotp]_\bF,\cB)\stackrel{{(\Psi,\psi)}}\rightleftarrows (\bE,[\cdotp,\cdots,\cdotp]_\bE,\cA)$, then we obtain a comorphism of  Leibniz-Rinehart algebra $(\mathcal{F},[\cdotp,\cdotp]_\mathcal{F},\cB)\stackrel{{(\Psi,\psi)}}\rightleftarrows (\mathcal{E},[\cdotp,\cdotp]_\mathcal{E},\cA)$.
	\\
	Conversely, let $(\mathcal{F},[\cdotp,\cdotp]_\mathcal{F},\cB)\stackrel{{(\Psi,\psi)}}\rightleftarrows (\mathcal{E},[\cdotp,\cdotp]_\mathcal{E},\cA)$ be a comorphism of  Leibniz-Rinehart algebra, then there exists a comorphism of $n$-Lie-Rinehart algebras   $(\bF,[\cdotp,\cdots,\cdotp]_\bF,\cB)\stackrel{{(\Psi,\psi)}}\rightleftarrows (\bE,[\cdotp,\cdots,\cdotp]_\bE,\cA)$.
\end{remark}

\begin{proposition}[\cite{ref5}]\label{prop4.4}
	Let $\bE,\bF$ be finitely generated projective $\cA$, $\cB$-modules respectively, and $\psi:\cA \to \cB$ be an algebraic homomorphism.
	\begin{enumerate}
		\item[1)]$\bE \otimes_\cA \cB$ is a finitely generated projective $\cB$-module, and
		\begin{eqnarray*}
			(\bE \otimes_\cA \cB)\wedge_\cA (\bE \otimes_\cA \cB)&=&(\wedge^2_\cA\bE)\otimes_\cA \cB\\
			&\cdots&
			\\
			\underbrace{(\bE \otimes_\cA \cB)\wedge_\cA\cdots \wedge_\cA(\bE \otimes_\cA \cB)}_{\mbox{k-copies}}&=&(\wedge^k_\cA\bE)\otimes_\cA \cB.
		\end{eqnarray*}
		\item[2)] The map $I:\bE \otimes_\cA \cB \to Hom_\cA(\bE^*_\cA,\cB)$, sending each $X\otimes_\cA b$ to
		$$I(X\otimes_\cA b): \xi \to \psi(<\xi,X>)b,\hspace*{0.5em}\forall \xi \in \bE^*_\cA,b\in \cB,$$
		is an isomorphism of $\cB$-modules. Similarly, we have
		\begin{eqnarray*}
			(\wedge^2_\cA\bE)\otimes_\cA \cB &\cong& Hom_\cA(\wedge^2_\cA\bE^*_\cA,\cB)\\
			&\cdots&\\
			(\wedge^k_\cA\bE)\otimes_\cA \cB &\cong& Hom_\cA(\wedge^k_\cA\bE^*_\cA,\cB).
		\end{eqnarray*}
		\item[3)] Let $\Psi: \bF \to \bE \otimes_\cA \cB$ be a $\cB$-map. There is an induced  $\cA$-map $\Psi^*: \bE^*_\cA \to \bF^*_\cB$, called the dual map of $\Psi$, such that
		$$<\Psi^*(\xi),Y>=<I\circ \Psi(Y),\xi>,\hspace*{0.5em}\forall \xi \in \bE^*_\cA,Y\in \bF.$$
		\item[3)] Let $\overline{\Psi}:\bE^*_\cA \to \bF^*_\cB$ be a map of $\cA$-module. There is a unique $\cB$-map $\Psi: \bF \to \bE\otimes_\cA\cB$ such that $\Psi^*=\overline{\Psi}$, ($\Psi(Y)=I^{-1}\circ <\overline{\Psi}(\cdotp),Y>$, for each $Y\in \bF$).
	\end{enumerate}
\end{proposition}
The proof of Proposition \ref{prop4.4} can see \cite{ref5}.

\begin{definition}\label{defi4.3}
	Let $\bE$ be an $n$-Lie-Rinehart algebra, which is a finitely generated projective $\cA$-module. We define a differential operator $d^{n-1}_\bE: \wedge_\cA^{k(n-1)}\bE_\cA^* \to \wedge_\cA^{(k+1)(n-1)}\bE_\cA^*$  as follows:
	\begin{enumerate}
		\item $<d^{n-1}_\bE a,x>
		=\widehat{\rho}(x)(a)=[x,a]_{\wedge^{n-1}_\cA\bE}$;
		\item $<d^{n-1}_\bE \xi,x\wedge_\cA y>
		=\widehat{\rho}(x)<\xi,y>-\widehat{\rho}(y)<\xi,x>-<\xi,[x,y]_{\wedge^{n-1}_\cA\bE}>$;
		\item \begin{eqnarray*}
			&&d^{n-1}_\bE (\xi_1\wedge_\cA \cdots \wedge_\cA \xi_m)(x_1\wedge_\cA\cdotp\cdotp\cdotp\wedge_\cA x_{m+1})\\
			&=&\sum_{i=1}^{m+1}(-1)^{i-1}\widehat{\rho}(x_i)<\xi_1\wedge_\cA \cdots \wedge_\cA \xi_m,x_1\wedge_\cA\cdotp\cdotp\cdotp\wedge_\cA \widehat{x_i}\wedge_\cA\cdotp\cdotp\cdotp\wedge_\cA x_{m+1}>\notag
			\\
			&&+\sum_{1\leq i<j\leq m+1}^{m+1}(-1)^{i+j}<\xi_1\wedge_\cA \cdots \wedge_\cA \xi_m,\notag\\
			&&[x_i,x_j]_{\wedge^{n-1}_\cA\bE}\wedge_\cA\cdotp\cdotp\cdotp\wedge_\cA \widehat{x_i}\wedge_\cA\cdotp\cdotp\cdotp\wedge_\cA \widehat{x_j}\wedge_\cA\cdotp\cdotp\cdotp\wedge_\cA x_{m+1}>,
		\end{eqnarray*}
	\end{enumerate}
	where $a\in \cA,x_1,\cdots,x_{m+1},y\in \wedge_\cA^{n-1}\bE_\cA$ and $\xi,\xi_1,\cdots,\xi_m\in \wedge_\cA^{n-1}\bE^*_\cA$.
\end{definition}
\begin{proposition}\label{prop4.6}
	Given two finitely generated projective $n$-Lie-Rinehart algebras  $(\bE,\cA)$ and $(\bF,\cB)$ over the algebras $\cA$ and $\cB$ respectively. Then the following statements are equivalent
	\begin{enumerate}
		\item[1)] $(\mathcal{F},[\cdotp,\cdotp]_{\wedge_\cB^{n-1}\bF:=\mathcal{F}},\cB)\stackrel{(\Psi,\psi)}\rightleftarrows (\mathcal{E},[\cdotp,\cdotp]_{\wedge_\cA^{n-1}\bE:=\mathcal{E}},\cA)$
		is a comorphism of Leibniz-Rinehart  algebras;
		\item[2)] the dual map $\Psi^*:\bE^*_\cA\to\bF^*_\cB$ satisfies
		\begin{equation}
			d^{n-1}_\bF \circ \Psi^*=\Psi^*\circ d^{n-1}_\bE,\hspace*{0.5em}\text{as a map} \wedge_\cA^{k(n-1)}\bE_\cA^* \to \wedge_\cA^{(k+1)(n-1)}\bF_\cA^*\hspace*{0.5em}(k\geq 0).
		\end{equation}	Here we regard $\Psi^*=\psi:\wedge^0_\cA\bE^*_\cA=\cA \to \wedge^0_\cB\bF^*_\cB=\cB$, and $\Psi^*$ naturally lifts to an $\cA$-map $\wedge^{k(n-1)}_\cA\bE^*_\cA \to \wedge^{k(n-1)}_\cB\bF^*_\cB$.
	\end{enumerate}
	Here the comorphism of Leibniz-Rinehart  algebras is defined by Remark \ref{remark4.3}.
\end{proposition}
\begin{proof}
	Since the Equation $<d^{n-1}_\bE a,x>
	=\widehat{\rho}(x)(a)=[x,a]_{\wedge^{n-1}_\cA\bE}$, then the second statement is equivalent to the following two conditions:
	\begin{enumerate}
		\item[1)] $d^{n-1}_\bF(\psi(a))=\Psi^*(d^{n-1}_\bE a),\hspace*{0.5em}\forall a\in \cA$;
		\item[2)] $d^{n-1}_\bF(\Psi^*(\xi))=\Psi^*(d^{n-1}_\bE(\xi) ),\hspace*{0.5em}\forall \xi\in \wedge^{n-1}_\cA \bE^*_\cA$.
	\end{enumerate}
	We now prove that these two conditions are equivalent to the first statement.
	
	Let $y\in \wedge^{n-1}_\cB\bF$ and $\Psi(y)=\sum_{k}x_{k}\otimes_\cA b_{k}$, for some $x_{k}\in \wedge^{n-1}_\cA\bE,b_{k}\in \cB$. By the condition $d^{n-1}_\bF(\psi(a))=\Psi^*(d^{n-1}_\bE a),\hspace*{0.5em}\forall a\in \cA$, we have
	\begin{eqnarray}
		&&[y,\psi(a)]_{\wedge^{n-1}_\cB\bF}\notag\\
		&=&<d^{n-1}_\bF(\psi(a)),y>\notag	\\
		&=&<\Psi^*(d^{n-1}_\bE(a)),y>\notag	\\
		&=&<d^{n-1}_\bE(a),I(\Psi(y))>\notag	\\
		&=&\sum_{k}\psi<d^{n-1}_\bE(a),x_{k}>b_{k}\notag	\\
		&=&\sum_{k}b_{k}\psi([x_k,a]_{\wedge^{n-1}_\cA\bE}).
	\end{eqnarray}
	Let $y_1,y_2\in \wedge^{n-1}_\cB \bF$ and $\Psi(y_i)=\sum_{k_i}x_{k_i}\otimes_\cA b_{k_i}$, for some $x_{k_i}\in \wedge^{n-1}_\cA\bE,b_{k_i}\in \cB,i=1,2$. Then we have
	\begin{eqnarray}
		&&<\Psi^*(d^{n-1}_\bE(\xi)),y_1\wedge_\cB y_2>\notag\\
		&=&<d^{n-1}_\bE(\xi),I(\Psi(y_1)\wedge_\cB \Psi(y_2))>\notag	\\
		&=&<d^{n-1}_\bE(\xi),I(\sum_{k_1k_2}(x_{k_1}\wedge_\cA x_{k_2})\otimes_\cA b_{k_1}b_{k_2})>\notag	\\
		&=&\sum_{k_1k_2}\psi(<d^{n-1}_\bE(\xi),x_{k_1}\wedge_\cA x_{k_2}>)b_{k_1}b_{k_2}> \notag	\\
		&=&\sum_{k_1k_2}\psi([x_{k_1},<\xi,x_{k_2}>]_{\wedge^{n-1}_\cA\bE}-[x_{k_2},<\xi,x_{k_1}>]_{\wedge^{n-1}_\cA\bE}-<\xi,[x_{k_1},x_{k_2}]_{\wedge^{n-1}_\cA\bE}>)b_{k_1}b_{k_2}\notag
		\\
		&=&\sum_{k_2}<d^{n-1}_\bF\circ \psi(<\xi,x_{k_2}>),y_1>b_{k_2}-\sum_{k_1}<d^{n-1}_\bF\circ \psi(<\xi,x_{k_1}>),y_2>b_{k_1}\notag
		\\
		&&\ -\sum_{k_1k_2}\psi(<\xi,[x_{k_1},x_{k_2}]_{\wedge^{n-1}_\cA\bE}>)b_{k_1}b_{k_2}.
	\end{eqnarray}
	On the other hand, we have
	\begin{eqnarray}
		&&<d^{n-1}_\bE(\Psi^*(\xi)),y_1\wedge_\cB y_2>\notag	\\
		&=&[y_1,<\Psi^*\xi,y_2>]_{\wedge^{n-1}_\cB\bF}-[y_2,<\Psi^*\xi,y_1>]_{\wedge^{n-1}_\cB\bF}-<\Psi^*\xi,[y_1,y_2]_{\wedge^{n-1}_\cB\bF}>\notag	\\
		&=&[y_1,\sum_{k_2}\psi<\xi,x_{k_2}>b_{k_2}]_{\wedge^{n-1}_\cB\bF}-[y_2,\sum_{k_1}\psi<\xi,x_{k_1}>b_{k_1}]_{\wedge^{n-1}_\cB\bF}-<\xi,I\circ \Psi [y_1,y_2]_{\wedge^{n-1}_\cB\bF}>\notag	\\
		&=&\sum_{k_2}<d^{n-1}_\bF\psi(<\xi,x_{k_2}>),y_1>b_{k_2}+\sum_{k_2}\psi(<\xi,x_{k_2}>)<d^{n-1}_\bF b_{k_2},y_1>\notag	\\
		&&-\sum_{k_1}<d^{n-1}_\bF\psi(<\xi,x_{k_1}>),y_2>b_{k_1}-\sum_{k_1}\psi(<\xi,x_{k_1}>)<d^{n-1}_\bF b_{k_1},y_2>\notag\\
		&&\ -<\xi,I\circ \Psi[y_1,y_2]_{\wedge^{n-1}_\cB\bF}>.
	\end{eqnarray}
	Thus, by condition $d^{n-1}_\bF (\Psi^*(\xi))=\Psi^*(d^{n-1}_\bE(\xi))$, we get
	\begin{eqnarray}
		&&<\xi,I\circ \Psi[y_1,y_2]_{\wedge^{n-1}_\cB\bF}>\notag\\
		&=&\sum_{k_1k_2}\psi(<\xi,[x_{k_1},x_{k_2}]_{\wedge^{n-1}_\cA\bE}>)b_{k_1}b_{k_2}+\sum_{k_2}\psi(<\xi,x_{k_2}>)<d^{n-1}_\bF b_{k_2},y_1>\notag	\\
		&&-\sum_{k_1}\psi(<\xi,x_{k_1}>)<d^{n-1}_\bF b_{k_1},y_2>.
	\end{eqnarray}
	Let
	$$z=\sum_{k_1k_2}[x_{k_1},x_{k_2}]_{\wedge^{n-1}_\cA\bE}\otimes_\cA b_{k_1}b_{k_2}+\sum_{k_2}x_{k_2}\otimes_\cA [y_1,b_{k_2}]_{\wedge^{n-1}_\cB\bF}-\sum_{k_1}x_{k_1}\otimes_\cA [y_2,b_{k_1}]_{\wedge^{n-1}_\cB\bF}.$$
	Then we have
	$$<\xi,I(\Psi[y_1,y_2]_{\wedge^{n-1}_\cB\bF}-z)>=0,\hspace*{0.5em}\forall \xi\in \wedge^{n-1}_\cA\bE^*_\cA.$$
	Since the map $I$ is an isomorphism, we obtain that $\Psi[y_1,y_2]_{\wedge^{n-1}_\cB\bF}-z$ A similar method shows that $2)$ implies $1)$.
\end{proof}

The graph of a pair  $(\bE,[\cdotp,\cdots,\cdotp]_\bE,\cA)\stackrel{{(\Psi,\psi)}}\rightrightarrows (\bF,[\cdotp,\cdots,\cdotp]_\bF,\cB)$ is defined by
$$Gr_{(\Psi,\psi)}:=\{x+\widehat{\Psi}(x)|x\in \bE \otimes_\cA \cB\}\subset (\bE\otimes_\cA \cB)\oplus \bF.$$
Here $\widehat{\Psi}$ is the $\cB$-map $\bE\otimes_\cA \cB\to \bF$ defined by $X\otimes_\cA \cB \to \Psi(X)b,\hspace*{0.5em}X\in \bE,b\in \cB$.

The graph of a pair   $(\bF,[\cdotp,\cdots,\cdotp]_\bF,\cB)\stackrel{{(\Psi,\psi)}}\rightleftarrows (\bE,[\cdotp,\cdots,\cdotp]_\bE,\cA)$ is given by
$$Gr_{(\Psi,\psi)}:=\{\Psi(Y)+Y|Y\in \bF\}\subset (\bE\otimes_\cA \cB)\oplus \bF.$$
Thus, we have the following theorem:
\begin{theorem}\label{thero4.7}
	Given two $n$-Lie-Rinehart  algebras $(\bE,[\cdotp,\cdots,\cdotp]_\bE,\cA)$ and $(\bF,[\cdotp,\cdots,\cdotp]_\bF,\cB)$ and an algebraic homomorphism $\psi: \cA \to \cB$.
	\begin{enumerate}
		\item[1)]$(\bE,[\cdotp,\cdots,\cdotp]_\bE,\cA)\stackrel{{(\Psi,\psi)}}\rightrightarrows (\bF,[\cdotp,\cdots,\cdotp]_\bF,\cB)$ is a morphism of $n$-Lie-Rinehart algebras if and only if its graph $Gr_{(\Psi,\psi)}$ is an $n$-Lie-Rinehart subalgebra (over $\cB$) of the $\psi$-sum $\bE\oplus_\psi \bF$.
		\item[2)] $(\bF,[\cdotp,\cdots,\cdotp]_\bF,\cB)\stackrel{{(\Psi,\psi)}}\rightleftarrows (\bE,[\cdotp,\cdots,\cdotp]_\bE,\cA)$ is a comorphism of $n$-Lie-Rinehart algebras if and only if its graph $Gr_{(\Psi,\psi)}$ is an $n$-Lie-Rinehart subalgebra (over $\cB$) of the $\psi$-sum $\bE\oplus_\psi \bF$.
	\end{enumerate}
\end{theorem}
\begin{proof}
	1)  Consider $x=X\otimes_\cA b$, Then $x+\widehat{\Psi}(x)=X\otimes_\cA b +b\Psi(X)\in Gr_{(\Psi,\psi)}$. By $(\ref{eq1.1})$ in Theorem $\ref{theorem 1.7}$, $x+\widehat{\Psi}(x)$ belongs to $\bE\oplus_\psi \bF$ if and only if
	\begin{eqnarray*}
		&&\psi([X_1,\cdotp,\cdotp \cdotp,X_{n-1},a]_\bE)b_1\cdots b_{n-1}\\
		&=&[b_1\Psi(X_1),\cdots,b_{n-1}\Psi(X_{n-1}),\psi(a)]_\bF\\
		&=&b_1\cdots b_{n-1}[\Psi(X_1),\cdots,\Psi(X_{n-1}),\psi(a)]_\bF
	\end{eqnarray*}
	holds, for all $X_i\otimes_\cA b_i+\widehat{\Psi}(X_i\otimes_\cA b_i)\in Gr_{(\Psi,\psi)},\hspace*{0.5em}\forall 2\leq i \leq n-1$. By $(\ref{eq2.1})$ in Definition \ref{defini 2.1} and the arbitrariness of $b_i$, the above condition holds if and only if the pair  $(\bE,[\cdotp,\cdots,\cdotp]_\bE,\cA)\stackrel{{(\Psi,\psi)}}\rightrightarrows (\bF,[\cdotp,\cdots,\cdotp]_\bF,\cB)$ satisfies Equation $\eqref{eq2.1}$. Next, we need only to prove that the graph $Gr_{(\Psi,\psi)}$ is closed under the $n$-bracket if and only if Equation \eqref{morphism defin 2} holds for the pair  $(\bE,[\cdotp,\cdots,\cdotp]_\bE,\cA)\stackrel{{(\Psi,\psi)}}\rightrightarrows (\bF,[\cdotp,\cdots,\cdotp]_\bF,\cB)$ . In fact, by the Proposition \ref{prop 1.8}, the condition is true.
	
	2) Similar to 1), the graph $Gr_{(\Psi,\psi)}$ is  also a subalgebra of $\bE\oplus_\psi \bF$.
	
	The proof is finished.
\end{proof}
\section{Morphisms and Comorphisms of $n$-Lie algebroids}
In this section, we give the definitions of morphisms and comorphisms of $n$-Lie algebroids.
It is proved that morphisms and comorphisms of $n$-Lie algebroids are equivalent to  comorphisms and morphisms of $n$-Lie-Rinehart algebras respectively.


Given a smooth map $\phi: M\to N$. Let $E_2$
be a vector bundle over $N$. We have the pull-back bundle $\phi^!E_2$ (over $M$). Thus, there exist an algebraic homomorphism $\psi=\phi^*:C^\infty(N)\to C^\infty(M)$. Let $E_1$ be a vector bundle over $M$ and   $\Phi_E: \phi^!E_2 \to E_1$ be a bundle map. We have the dual bundle map
\begin{equation}
	\xymatrix{
		E_1^* \ar[r]^{\Phi_{E^*}} \ar[d]^{p_1} & E_2^* \ar[d]^{p_2}\\
		M \ar[r]^{\phi} & N
	}
\end{equation}
Obviously, it induces a map
\begin{equation}\label{eq 1.2}
	\Phi_E^*: \Gamma(E_2) \to  \Gamma(E_1).
\end{equation}
The vector bundle comorphism have a different definition from \cite{ref21}.
\begin{definition}\label{defini 1.1}
	A vector bundle comorphism, depicted by a diagram
	\begin{equation}
		\xymatrix{
			E_1 \ar@{-->}[r]^{\Phi_{E}} \ar[d] & E_2 \ar[d]\\
			M \ar[r]^{\phi} & N
		}
	\end{equation}
	is given by a base map $\phi: M\to N$ together with a family of linear maps (going in the
	'opposite' direction)
	$$\Phi_{E}: (E_2)_{\phi(x)} \to (E_1)_x$$
	depending smoothly on $x$, in the sense that the resulting map $\phi^*E_2 \to E_1$ is smooth.
\end{definition}
The above vector bundle comorphism is equivalent to a vector bundle map $\Phi_{E}: \phi^!E_2 \to E_1$.
\begin{definition}\label{defini3.2}
	Let $(E_1,M,\rho_1,[\cdotp,\cdots,\cdotp]_1)$ and $(E_2,N,\rho_2,[\cdotp,\cdots,\cdotp]_2)$ be two $n$-Lie algebroids over bases $M$ and $N$ respectively. Given a smooth map $\phi: M\to N$, and a vector bundle morphism $\Phi_E: \phi^!E_2 \to E_1$, if it satisfies
	\begin{enumerate}
		\item[(i)] $\phi_*\circ \rho_1\circ \Phi_E^*=\rho_2$,
		\item[(ii)] the pullback map $(\ref{eq 1.2})$ preserves $n$-brackets,
	\end{enumerate}
	then we call $\Phi_E: \phi^!E_2 \to E_1$  a \textbf{comorphism of $n$-Lie algebroids},
	written as
	$$(E_2,N,[\cdotp,\cdots,\cdotp]_2)\stackrel{(\Phi_{E},\phi)} \rightleftarrows(E_1,M,[\cdotp,\cdots,\cdotp]_1).$$
	In particular, if $\phi$ is surjective and $\Phi_E$ is injective, then we call $(E_2,N)$ a co-subalgebroid of $(E_1,M)$.
\end{definition}
Obviously, this condition $(i)$ is not automatic hold. For example, let $N=M$, with $\phi$ the
identity map, let $E_2=TM$ be the tangent bundle.  Let $E_1=0$ be the trivial $n$-Lie algebroid with zero anchor map and zero $n$-bracket. Let
$X_1,\cdots,X_{n-1}\in \Gamma(TM)$ be some non-zero vector fields. Then there is a unique $n$-Lie algebroid comorphism  $\Phi_{E}:\phi^!E_2 \to 0$ covering
$\phi=id_M$; the pull-back map on sections is the zero map, and in particular preserves brackets.
But the condition $(i)$ would tell us $0 \sim_\phi \rho_2(X_1\wedge \cdots X_{n-1})$, i.e. $\rho_2(X_1\wedge \cdots \wedge X_{n-1})=0$.
\begin{remark}
	For every open set of all $x\in M$, where the pullback map $\Phi_{E}^*: \Gamma((E_2)_{\phi(x)})\to \Gamma((E_1)_x)$ is non-zero, condition $(i)$ is automatic. First, for every sections $\sigma_1,\cdots,\sigma_n$ of $E_2$ and $f\in C^\infty(N)$, we have $\Phi_{E}^*[\sigma_1,\cdotp,\cdotp,\cdotp,f\sigma_n]=[\Phi_{E}^*\sigma_1,\cdotp,\cdotp,\cdotp,(\phi^*f)\Phi_{E}^*\sigma_n]$. Using the Leibnitz rule, one obtain a formula
	$$(\phi^*(\rho_2(\sigma_1\wedge \cdots \wedge \sigma_{n-1})f)-\rho_1(\Phi_{E}^*(\sigma_1\wedge \cdots \wedge \sigma_{n-1}))(\phi^*f))\Phi_{E}^*\sigma_n=0.$$
	This proves that $\phi^*(\rho_2(\sigma_1\wedge \cdots \wedge \sigma_{n-1})f)=\rho_1(\Phi_{E}^*(\sigma_1\wedge \cdots \wedge \sigma_{n-1}))(\phi^*f)$ at all point $x\in M$ where $\Phi_{E}^*\sigma_n|_x \neq 0$ for every $\sigma_1,\cdots, \sigma_{n-1} \in \Gamma(E_2)$.
\end{remark}
\begin{corollary}\label{example5.4}
	Let $\psi=\phi^*:C^\infty(N)\to C^\infty(M)$. In Definition \ref{defini3.2}, $(1)$ is equivalent to
	$$\psi([\sigma_1,\cdotp\cdotp \cdotp,\sigma_{n-1},f]_{E_2})=[\Phi_{E}^*(\sigma_1),\cdots,\Phi_{E}^*(\sigma_{n-1}),\psi(f)]_{E_1},\hspace*{0.5em}\forall \sigma_1,\cdots,\sigma_{n-1}\in E_2,f\in C^\infty(N).$$
	(\text{C.f.} relation $(1)$ of Definition \ref{defini 2.1}.) Thus, the \textbf{comorphism}  of $n$-Lie algebroids in Definition \ref{defini3.2} is
	equivalent to the fact that
	$$(\Gamma(E_2),C^\infty(N))\stackrel{(\Phi_{E}^*,\psi)}\rightrightarrows (\Gamma(E_1),C^\infty(M))$$
	is a \textbf{morphism} of $n$-Lie-Rinehart algebras.
\end{corollary}
We remark that, the first condition can be restated as: for each $Y_1,\cdots ,Y_{n-1} \in \Gamma(E_2)$, the vector field $\rho_1(\Phi_E^*(Y_1\wedge \cdots \wedge Y_{n-1}))$
is $\phi$-related to $\rho_2(Y_1\wedge \cdots \wedge Y_{n-1})$. We denote by  $\mathcal{LA^\vee}$ the category of $n$-Lie algebroids of rank $n$ and $n$-Lie algebroid comorphisms.
\begin{definition}[\cite{ref8}]
	Let $(M,\pi_1)$ and $(N,\pi_2)$ be two manifolds with $n$-vector fields. A smooth
	map $\phi: M\to N$ is called $(\pi_1,\pi_2)$-map if the induced brackets on functions satisfy:
	\begin{equation*}
		\{\phi^*f_1,\cdots,\phi^*f_n\}_1=\phi^*\{f_1,\cdots,f_n\}_2,
	\end{equation*}
	for all $f_1,\cdots,f_n\in C^\infty(N)$, or equivalently, $\phi_*\pi_1=\pi_2$.
	A $(\pi_1,\pi_2)$-map $\phi:(M,\pi_1)\to (N,\pi_2)$ between Nambu-Poisson manifolds of the same order $n$ is called a Nambu-Poisson map.
\end{definition}
Now, let $\mathcal{VB}_{Nambu}$  be the category of vector bundles with linear Nambu-Poisson structures of rank $n$; morphisms
in this category are vector bundle maps that are also Nambu-Poisson maps.  The following
result shows that there is category equivalence  between
$\mathcal{VB}_{Nambu}$  and $\mathcal{LA^\vee}$.
For any section $\sigma\in\Gamma(E)$, let $\varphi_\sigma \in C^\infty(M)$ be the corresponding linear function on the dual
bundle $E^*$. 	
\begin{theorem}\label{theo5.5}
	Let
	$(E_1,M,\rho_1,[\cdotp,\cdots,\cdotp]_1)$ and $(E_2,N,\rho_2,[\cdotp,\cdots,\cdotp]_2)$ be two $n$-Lie algebroids of rank $n$ over bases $M_1$ and $M_2$ respectively. A vector bundle morphism $\Phi_E: \phi^!E_2 \to E_1$ is an $n$-Lie algebroid comorphism if and only if the dual map $\Phi_{E^*}:E_1^* \to E_2^*$
	is a Nambu-Poisson map.
\end{theorem}
\begin{proof}
	Let $p_1: E_1^* \to M$ and $p_2: E_2^* \to N$ be two dual bundles correspondence to $E_1$ and $E_2$ respectively. To simplify notation, we denote all the pull-back maps $\phi^*,\Phi_E^*,\Phi^*_{E^*}$ by $\Phi^*$.
	For any vector bundle morphism $\Phi_E: \phi^!E_2 \to E_1$, and $\sigma \in \Gamma(E_2)$, we have that
	\begin{equation}
		\varphi_{\Phi^*\sigma}=\Phi^*\varphi_\sigma.
	\end{equation}
	Given sections $\sigma_1,\cdots,\sigma_n\in \Gamma(E_2)$ and  functions $f_1,\cdots,f_n\in C^\infty(M_2)$, for all $ 0\leq k \leq n-2,$ we have
	\begin{equation}\label{eq1.4}
		\Phi^*\{\varphi_{\sigma_1},\cdotp\cdotp\cdotp,\varphi_{\sigma_k},p_2^*f_{k+1},\cdotp,\cdotp,\cdotp,p_2^*f_n\}_2=0
		=\{\Phi^*\varphi_{\sigma_1},\cdotp,\cdotp,\cdotp,\Phi^*\varphi_{\sigma_k},\Phi^*p_2^*f_{k+1},\cdotp\cdotp\cdotp,\Phi^*p_2^*f_n\}_1, \end{equation}
	\begin{equation}\label{eq1.5}
		\varphi_{\Phi^*[\sigma_1,\cdots,\sigma_n]_2}=\Phi^*\varphi_{[\sigma_1,\cdots,\sigma_n]_2}=\Phi^*\{\varphi_{\sigma_1},\cdots,\varphi_{\sigma_n}\}_2,
	\end{equation}
	\begin{equation}\label{eq1.6}
		\varphi_{[\Phi^*\sigma_1,\cdots,\Phi^*\sigma_n]_1}=\{\varphi_{\Phi^*\sigma_1},\cdots,\varphi_{\Phi^*\sigma_n}\}_1=\{\Phi^*\varphi_{\sigma_1},\cdots,\Phi^*\varphi_{\sigma_n}\}_1,
	\end{equation}
	\begin{equation}\label{eq1.7}
		p_1^*\Phi^*(\rho_2(\sigma_1\wedge\cdots \wedge \sigma_{n-1})f_1)=\Phi^*p_2^*(\rho_2(\sigma_1\wedge\cdots \wedge \sigma_{n-1})f_1)=\Phi^*\{\varphi_{\sigma_1},\cdots,\varphi_{\sigma_{n-1}},p_2^*f_1\}_2,
	\end{equation}
	\begin{equation}\label{eq1.8}
		p_1^*(\rho_1(\Phi^*\sigma_1\wedge\cdots \wedge\Phi^*\sigma_{n-1})\Phi^*f_1)=\{\varphi_{\Phi^*\sigma_1},\cdots,\varphi_{\Phi^*\sigma_{n-1}},p_1^*\Phi^*f_1\}_1
		=\{\Phi^*\varphi_{\sigma_1},\cdots,\Phi^*p_2^*f_1\}_1.
	\end{equation}
	The Equation $\eqref{eq1.4}$  holds because of the local coordinates of $\pi_1$ and $\pi_2$ on $E_1^*$ and $E_2^*$ respectively. Therefore, $\Phi_E$ being an $n$-Lie algebroid comorphism is equivalent to the equality of the left hand sides of Equations $\eqref{eq1.5}$, $\eqref{eq1.6}$
	and equality of the left hand sides of Equations $\eqref{eq1.7}$, $\eqref{eq1.8}$, while $\Phi^*$ being a Nambu-Poisson
	map is equivalent to the equality of the corresponding right hand sides.
\end{proof}
\begin{remark}
	It is known that under some connectedness and simply connectedness assumption, any Lie bialgebra integrates to a Poisson-Lie group, and any Lie bialgebroid integrates to a Poisson groupoid. These results does not hold in the context of Nambu structures of order $\geq 3$.
	Therefore, we  not consider the integral problem of $n$-Lie algebroids in this paper.
\end{remark}

Next we define $n$-Lie algebroid morphism $H_E: E_1 \to E_2$.
\begin{definition}\label{defini 5.6}
	Let $(E,M,\rho,[\cdotp,\cdots,\cdotp])$ be an $n$-Lie algebroid, and $H\subseteq E$ a vector subbundle along $N\subseteq M$. We say that a vector subbundle $H$ is an $n$-Lie subalgebroid when it has the following properties:
	\begin{enumerate}
		\item[a.]If $\sigma_1|_N,\cdots,\sigma_n|_N\in \Gamma(H)$, then we have $[\sigma_1,\cdots,\sigma_n]|_N\in \Gamma(H)$, where $\sigma_1,\cdots,\sigma_n \in\Gamma(E)$,
		\item[b.]$\rho(\wedge^{n-1}H)\subseteq TN$.
	\end{enumerate}
	Therefore,  an $n$-Lie subalgebroid is itself an $n$-Lie algebroid.
\end{definition}
\begin{proposition}
	Let $H\subseteq E$ is an $n$-Lie subalgebroid along $N\subseteq M$, then $H$ has an $n$-Lie algebroid structure, with  anchor the restriction of $\rho: \wedge^{n-1}E\to TN$, and with the unique bracket
	such that
	\begin{equation}
		[\sigma_1|_N,\cdots,\sigma_n|_N]_N=[\sigma_1,\cdots,\sigma_n]|_N
	\end{equation}
	whenever $\sigma_1|_N,\cdots,\sigma_n|_N\in \Gamma(H)$.
\end{proposition}
\begin{proof}
	To prove that this bracket is well-defined, we have to show that $[\sigma_1,\cdots,\sigma_n]|_N=0$ whenever
	$\sigma_n|_N=0$. Let $\sigma_n=\sum_{j}f_j^n\sigma_j^n$ where $f_j^n\in C^\infty(M)$ vanish on $N$. Thus, we have that
	$$[\sigma_1,\cdots,\sigma_n]|_N=\sum_{j}f_j^n|_N[\sigma_1, \cdots,  \sigma_{n-1},\sigma_j^n]|_N+\sum_{j}(\rho(\sigma_1\wedge \cdots \wedge \sigma_{n-1})f_j^n)|_N\sigma_j^n|_N=0,$$
	where we used that $\rho(\sigma_1\wedge \cdots \wedge \sigma_{n-1})f_j^n=0$, since $\rho(\sigma_1\wedge \cdots \wedge \sigma_{n-1})\in TN$ and the $f_j^n$ vanish on $N$.
\end{proof}
Now we use $n$-Lie subalgebroids  to define morphisms of $n$-Lie algebroids.
\begin{definition}\label{defi5.8}
	Given two $n$-Lie algebroids $(E_1,M,\rho_1,[\cdotp,\cdots,\cdotp]_1)$ and $(E_2,N,\rho_2,[\cdotp,\cdots,\cdotp]_2)$, a vector bundle map
	$$\Phi_E: E_1\to E_2$$
	is an \textbf{$n$-Lie algebroid morphism}, written $(E_1,M,[\cdotp,\cdots,\cdotp]_1)\stackrel{(\Phi_{E},\phi)}\rightrightarrows(E_2,N,[\cdotp,\cdots,\cdotp]_2)$, if its graph $Gr(\Phi_E)\subseteq E_2\times E_1^{-(n-1)}$ is an $n$-Lie subalgebroid along $Gr(\phi)$.
\end{definition}
$E_1^{-(n-1)}$ is $E_1$ as a
vector space, but the $n$-Lie bracket on section spaces $\Gamma(A)$ is given by $(-1)^{n-1}$ the bracket on $E_1$, and with $(-1)^{n-1}$
the anchor of $E_1$.
The category of $n$-Lie algebroids of rank $n$ with morphisms will be denoted by $\mathcal{LA}$.
Having defined the category $\mathcal{LA}$, it is natural to ask what corresponds to it on the dual side,
in terms of the linear Nambu-Poisson structures on vector bundles. The answer will have to wait until
we have the notion of  Nambu-Poisson relations.
\begin{theorem}
	With the above notations, $(E_1,M,[\cdotp,\cdots,\cdotp]_1)\stackrel{(\Phi_{E},\phi)}\rightrightarrows(E_2,N,[\cdotp,\cdots,\cdotp]_2)$ is a morphism of $n$-Lie
	algebroids if and only if
	\begin{enumerate}\label{th5.9}
		\item[1)]\begin{equation}\label{eq5.11}
			\rho_2\circ \Phi_{E}=\phi_*\circ \rho_1
		\end{equation}
		\item[2)]
		\begin{equation}\label{eq5.12}
			\Phi_E([\sigma_1,\cdots,\sigma_n])=[\Phi_E(\sigma_1),\cdots,\Phi_E(\sigma_n)],
		\end{equation}
		for all  $\sigma_1,\cdots,\sigma_n\in \Gamma(E_1)$.
	\end{enumerate}
\end{theorem}
The Equation \eqref{eq5.12} can be written as local expression as follows. Let $\Psi_E^!(\sigma_i)=\sum_{k_i}f_{k_i}\tau_{k_i},\hspace*{0.5em}1\leq i\leq n$, for some $f_{k_i}\in C^\infty(M),\tau_i\in \Gamma(E_2)$, where $\Psi_E^!$
is the induced bundle map $E_1\to \phi^!E_2$, then
\begin{eqnarray*}
	\Psi_E^!([\sigma_1,\cdots,\sigma_n])=\sum_{k_1\cdotp\cdotp\cdotp k_n}f_{k_1}\cdotp\cdotp\cdotp f_{k_n}[\tau_{k_1},\cdots,\tau_{k_n}]+\sum_{   i}\sum_{k_i}(-1)^{n-i}[\sigma_1,\cdots,\widehat{\sigma_i},\cdots,\sigma_n,f_{k_i}].
\end{eqnarray*}
\begin{proof}
	Assume that $(E_1,M,[\cdotp,\cdots,\cdotp]_1)\stackrel{(\Phi_{E},\phi)}\rightrightarrows(E_2,N,[\cdotp,\cdots,\cdotp]_2)$ is a morphism of $n$-Lie algebroids. Then $Gr(\Phi_{E})\subseteq E_2\times E^{-(n-1)}_1$ is a subalgebroid along $Gr(\phi)$. By the definition of $n$-Lie subalgebroids, we have:
	\begin{enumerate}
		\item[1)]for all $\sigma_1,\cdots,\sigma_{n-1}\in \Gamma(E_1)$ and $\Phi_{E}(\sigma_1),\cdots,\Phi_{E}(\sigma_{n-1})\in\Gamma(E_2)$,
		\begin{eqnarray*}
			&&\rho((\Phi_{E}(\sigma_1),\sigma_1),\cdots,(\Phi_{E}(\sigma_{n-1}),\sigma_{n-1}))\notag
			\\
			&=&(\rho_2(\Phi_{E}(\sigma_1)\wedge\cdotp\cdotp\cdotp\wedge\Phi_{E}(\sigma_{n-1})),\rho_1(\sigma_1\wedge\cdotp\cdot\cdot\wedge\sigma_{n-1}))\in TGr(\phi).
		\end{eqnarray*}
		Then by the definition of $TGr(\phi)$, we have the following expression:
		\begin{equation}\label{eq5.13}
			\rho_2\circ(\Phi_{E}(\sigma_1)\wedge\cdotp\cdotp\cdotp\wedge\Phi_{E}(\sigma_{n-1})) =\phi_*\circ \rho_1(\sigma_1\wedge\cdotp\cdot\cdot\wedge\sigma_{n-1})
		\end{equation}
		The Equation \eqref{eq5.13} is equivalent to the equation \eqref{eq5.11}.
		\item[2)]for all $(\Phi_{E}(\sigma_1),\sigma_1),\cdots,(\Phi_{E}(\sigma_{n}),\sigma_{n})\in \Gamma(Gr(\Phi_{E}))$,
		\begin{eqnarray*}
			&&[(\Phi_{E}(\sigma_1),\sigma_1),\cdots,(\Phi_{E}(\sigma_{n}),\sigma_{n})]
			\\
			&=&([\Phi_{E}(\sigma_1),\cdots,\Phi_{E}(\sigma_n)]_2,(-1)^{n-1}[\sigma_1,\cdots,\sigma_n]_1)\in \Gamma(E_2\times E^{-(n-1)}_1).
		\end{eqnarray*}
		Then by the definition of $Gr(\Phi_{E})$, we get the following expression:
		\begin{equation}\label{eq5.14}
			[\Phi_{E}(\sigma_1),\cdots,\Phi_{E}(\sigma_n)]_2=\Phi_{E}([\sigma_1,\cdots,\sigma_n]_1)
		\end{equation}
		The Equation \eqref{eq5.14} is exactly the Equation \eqref{eq5.12}.
	\end{enumerate}
	
	Conversely, the proof is obvious.
\end{proof}
\begin{corollary}\label{example5.10}
	Let $\psi=\phi^*:C^\infty(N)\to C^\infty(M)$. In Theorem \ref{th5.9}, $(1)$ is equivalent to
	$$[\sigma_1,\cdots,\sigma_{n-1},\psi(g)]_1=\sum_{k_1,\cdots,k_n}f_{k_1}\cdotp\cdotp\cdotp f_{k_n}\psi([\tau_1,\cdots,\tau_{n-1},g]_2),\hspace*{0.5em}\forall \sigma_1,\cdots,\sigma_{n-1}\in E_1,g\in C^\infty(N).$$
	(\text{C.f.} relation $(1)$ of Definition \ref{equivalent morphism}.) Thus, the \textbf{morphism} of $n$-Lie algebroid in Theorem (\ref{th5.9}) is equivalent to the fact that
	$$(\Gamma(E_1),C^\infty(M))\stackrel{(\Phi_{E}^!,\psi)}\rightleftarrows (\Gamma(E_2),C^\infty(N))$$
	is a \textbf{comorphism} of $n$-Lie-Rinehart algebras.
\end{corollary}

\section{Applications to Nambu-Poisson manifolds}

In this section,  we revisit some notations and  property about Nambu-Poisson submanifolds and  coisotropic submanifold of Nambu-Poisson manifolds.
Then we introduce the concept of Nambu-Poisson relation which is a generalization of Poisson relation introduced by Weinstein in \cite{ref26}.
Finally, we prove that $\Phi_{E}: E_1 \to E_2$ is an $n$-Lie algebroid morphism if and only if the dual comorphism $\Phi_{E^*}: E_1^* \dashrightarrow E_2^*$ is a Nambu-Poisson relation.
Thus we obtain that there is a category equivalence between $\mathcal{VB}^\vee_{Nambu}$  and $\mathcal{LA}$, see Theorem \ref{theo6.10}.

\subsection{Nambu-Poisson submanifolds}
A submanifold $N\subseteq M$ is called \textbf{Nambu-Poisson submanifold} if the Nambu-Poisson tensor $\pi$ is everywhere tangent to $N$, that is $\pi_x\in \wedge^nT_xN\subseteq \wedge^nT_xM$. Thus, we define an $n$-vector field $\pi_N\in \wedge^nTN$, satisfying the following property:
$$\pi_N \sim_i \pi$$
where $i: N\to M$ is the inclusion. Therefore, the corresponding Nambu-Poisson bracket $\{\cdotp,\cdots,\cdotp\}_N$ is given by
$$\{i^*f_1,i^*f_2,\cdots,i^*f_n\}_N=i^*\{f_1,f_2,\cdots,f_n\}.$$
The Jacobi identity for $\pi_N$ follows from that for $\pi$.
We have the following easy observation for Nambu-Poisson submanifolds of  Nambu-Poisson
manifolds.
\begin{proposition}
	The following are equivalent:
	\begin{enumerate}
		\item[(a)]\label{pro2.1 a} $N$ is a Nambu-Poisson submanifold.
		\item[(b)] $\pi^\sharp(\wedge^{n-1}T^*M|_N)\subseteq TN$.
		\item[(c)] $\pi^\sharp(\wedge^{n-2}T^*M|_N\wedge (TN)^\circ)=0$.
		\item[(d)] All Hamiltonian vector fields $X_{f_1\cdots f_{n-1}}$, $f_1,\cdots,f_{n-1}\in C^\infty(M)$ are tangent to $N$.
		
		\item[(e)] When $N$ is a closed embedded submanifold, these conditions (a)-(d) are also equivalent to
		the following property: the vanishing ideal of $N$
		$$\mathcal{I}(N):=\{f\in C^\infty(M)|f|_N\equiv0\}$$
		is an $n$-Lie algebra ideal; i.e., $\{f_1,\cdots,f_{i-1},g,f_{i+1},\cdots,f_n\}\in \mathcal{I}(N)$ whenever $f_j\in C^\infty(M)$ and $g\in \mathcal{I}(N)$.
	\end{enumerate}
\end{proposition}
\begin{proof}
	If $i : (N,\pi_N)\hookrightarrow (M,\pi)$ is a Nambu-Poisson submanifold, then $\pi_N$ is $i$-related to $\pi$:
	$$(i_*)_x(\pi_{N,x})=\pi_x,\hspace*{1em}\forall x\in N.$$
	This is equivalent to
	\begin{equation}
		(i_*)_x\circ \pi_{N,x}^\sharp\circ (i^*)_x=\pi^\sharp_x.
	\end{equation}
	Since $i_*$ is injective, this proves that $\pi_N$ is unique. Obviously, if $(N,\pi_N)$ is a Nambu-Poisson submanifold, then we have $\pi^\sharp(\wedge^{n-1}T^*M|_N)\subseteq TN$.
	
	Next, let $i :N  \hookrightarrow M$ be a submanifold such that  $\mathrm{Im}\pi^\sharp_x\subseteq (i_*)_x(TN)$. We claim that there exists a unique smooth $n$-vector field $\pi_N$ on $N$ such
	that $(a)$ holds. In fact, since $\mathrm{Im} \pi^\sharp_x\subseteq (i_*)_x(TN)$, it is enough to check that for
	any $\alpha_1\wedge \cdots \wedge \alpha_{n-1}\in \wedge^{n-2}T^*_xM\wedge (T_xN)^\circ=ker((i^*)_x)$ we have $\pi^\sharp_x(\alpha)=0$. By the $n$-skew-symmetry, for any $\alpha_n\in T^*_xM$ we have that
	$$<\pi^\sharp_x(\alpha_1\wedge \cdots \wedge \alpha_{n-1}),\alpha_n>=(-1)^{n-j}<\pi^\sharp_x(\alpha_1\wedge \cdots \wedge \alpha_{j-1}\wedge \alpha_n\wedge \alpha_{j+1}\wedge \cdots \wedge  \alpha_{n-1}),\alpha_j>=0.$$
	The $n$-skew-symmetry for $\pi_N$ follows from that for $\pi$. The smoothness of $\pi_N$ is
	automatically holds.
	
	The Schouten brackets of $i$-related multivector fields are also $i$-related:
	$$(i_*)_x([\pi_N,\pi_N])_x=[\pi,\pi]_{i(x)}=0.$$
	Therefore,  if (i) holds, then $N$ has a unique Nambu-Poisson structure such that it is a Nambu-Poisson submanifold.
	
	For the equivalence between $(b)$ and $(c)$, by using
	$$<\pi^\sharp_x(\alpha_1\wedge\cdots \wedge \alpha_{n-1}),\alpha_{n}>=(-1)^{n-j}<\pi^\sharp_x(\alpha_1\wedge \cdots \wedge \alpha_{j-1}\wedge \alpha_n\wedge \alpha_{j+1}\wedge \cdots \wedge  \alpha_{n-1}),\alpha_j>,$$
	with $\alpha_1,\cdots,\alpha_{n-1}\in T^*_xM$ and $\alpha_n\in (T_xN)^\circ$, we get
	$$\pi^\sharp_x(\wedge^{n-1}T^*_xM)\subseteq T_xN\hspace*{1em} \Leftrightarrow \hspace*{1em} \pi^\sharp_x(\wedge^{n-2}T^*_xM\wedge (T_xN)^\circ)=0.$$
	
	The equivalence between $(b)$ and $(d)$ is obvious.
	
	Let $N$ is an embedded submanifold.
	If $(d)$ holds, then the functions vanishing on $N$ are an $n$-Lie algebra ideal since $g|_N=0$ implies that
	$$\{f_1,\cdots,f_{i-1},g,f_{i+1},\cdots,f_n\}|_N=(-1)^{n-i}\{f_1,f_2,\cdots,f_{i-1},f_{i+1},\cdots,f_n,g\}|_N=0.$$
	Since $X_{f_1\cdots f_{i-1}f_{i+1}\cdots f_n}$ is tangent to $N$. This proves $(e)$. Conversely, if $(e)$ holds, such that
	$$\{f_1,f_2,\cdots,f_{i-1},f_{i+1},\cdots,f_n,g\}|_N=(-1)^{n-i}\{f_1,\cdots,f_{i-1},g,f_{i+1},\cdots,f_n\}|_N=0$$
	whenever $g|_N=0$. It follows that $<dg,X_{f_1\cdots f_{i-1}f_{i+1}\cdots f_n}>|_N=X_{f_1\cdots f_{i-1}f_{i+1}\cdots f_n}(g)|_N=0$, whenever $g|_N=0$. The differentials $dg|_N$ for $g|_N=0$ $(TN)^\circ$, hence this implies that $X_{f_1\cdots f_{i-1}f_{i+1}\cdots f_n}|_N \in \Gamma(TN)$, which gives $(d)$.
\end{proof}
\subsection{Coisotropic submanifolds}
In the subsection, we introduce coisotropic submanifolds of  Nambu-Poisson manifolds and give a Nambu-Poisson relation. The coisotropic submanifold of  Nambu-Poisson manifolds has introduced in \cite{ref8} through the closed embedded submanifold. However, under local condition, the embedded submanifold is equivalent to the immersion submanifold. In the subsection, a submanifold is an immersion submanifold unless otherwise specified. The following proposition from \cite{ref8}.
\begin{proposition}[\cite{ref8}]
	The following are equivalent:
	\begin{enumerate}
		\item[(a)] $\pi^\sharp(\wedge^{n-1}(TN)^\circ)\subseteq TN$.
		\item[(b)] For every $f_1,\cdots,f_{n-1}$ such that $f_i|_N=0,\hspace*{0.5em}\forall i=1,\cdots,n-1$, the Hamiltonian vector filed $X_{f_1\cdots f_{n-1}}$ is tangent to $N$.
		\item[(c)] The space of functions $f$ with $f|_N=0$ are an $n$-Lie subalgebra (or Nambu-Poisson subalgebra) under the Nambu-Poisson bracket.
	\end{enumerate}
\end{proposition}
A submanifold $N\subseteq M$ is called a coisotropic submanifold if it satisfies any of these equivalent
conditions.

In \cite{ref8}, the authors give the property of coisotropic submanifold with respect to any multivector field. Now, we give a similar proposition.
\begin{proposition}
	Let $(M_1,\pi_1)$ and $(M_2,\pi_2)$ be two manifolds with $n$-vector  fields $\pi_1$ and $\pi_2$ respectively. Let $\Phi: M_1 \to M_2$ be a smooth map. Then  $\Phi_*\pi_1=\pi_2$ if and only if its graph
	$$Gr(\Phi)=\{(\Phi(m_1),m_1)|m_1\in M_1\}$$
	is a coisotropic submanifold of $M_2 \times M_1$ with respect to $\pi_2\oplus(-1)^{n-1}\pi_1$.
\end{proposition}
\begin{proof}
	Note
	that, a tangent vector to the graph consist of a pair $(\Phi_*X_{m_1},X_{m_1})$, where $m_1\in M_1,X_{m_1}\in T_{m_1}M_1$. Therefore, $(TGr(\Phi))^\circ$ consist of a pair of covectors $(\alpha,-\Phi^*\alpha)$, where $\alpha \in T^*_{\Phi(m_1)}M_2$.
	Thus, $Gr(\Phi)$ is a coisotropic submanifold of $M_2\times M_1$ with respect to $\pi_2\oplus(-1)^{n-1}\pi_1$ if and only if $\pi_2^\sharp \times (-1)^{n-1}\pi_1$ map $(\alpha_1,-\Phi^*\alpha_1)\wedge \cdots \wedge (\alpha_{n-1},-\Phi^*\alpha_{n-1})$ into $TGr(\Phi)$, for all $\alpha_1,\cdots,\alpha_{n-1} \in T^*_{\Phi(m_1)}M_2$ and $m_1\in M_1$.  In other words,
	$\Phi_*\pi_1=\pi_2$.
	The proof is finished.
\end{proof}

Let $\pi_1,\pi_2$ be two Nambu-Poisson tensors on $M_1$ and $M_2$ respectively. Then the map $\Phi$ is a Nambu-Poisson map.
In \cite{ref26}, Weinstein introduced the notion of Poisson relation. In a similar fashion for Nambu-Poisson manifolds, we can give the following definition:
\begin{definition}
	Let $M_1,M_2$ be two Nambu-Poisson manifolds. A {\bf Nambu-Poisson relation} from $M_1$ to $M_2$ is a coisotropic submanifold $N\subseteq M_2\times M_1^{(-1)^{n-1}}$, where $M_1^{(-1)^{n-1}}$ is $M_1$ with the Nambu-Poisson structure $(-1)^{n-1}\pi_1$.
\end{definition}
Nambu-Poisson relations are regarded as "comorphism". We will thus write
$$N: M_1\dashrightarrow M_2$$
for a submanifold $N\subseteq M_2\times M_1$ seen as such a "morphism". But, we need consider the compatible condition for the Nambu-Poisson relation: Given submanifolds
$N\subseteq M_2\times M_1$ and $H\subseteq M_3\times M_2$, the composition $H\circ N$ need not to be a submanifold.
\begin{definition}
	We say that two relations $N: M_1\dashrightarrow M_2$ and $H: M_2\dashrightarrow M_3$ (given by submanifolds $N\subseteq M_2\times M_1$ and $H\subseteq M_3\times M_2$) have $\mathbf{clean\hspace*{0.5em} composition}$ if
	\begin{enumerate}
		\item[(a)] $H\circ N$ is a submanifold;
		\item[(b)] $T(H\circ N)=TH\circ TN$ $\mathbf{fiberwise}$.
	\end{enumerate}
\end{definition}
By $(b)$,  for all $m_i\in M_i$ with $(m_3,m_2)\in H$ and $(m_2,m_1)\in N$, we get that
$$T_{(m_3,m_1)}(H\circ N)=T_{(m_3,m_2)}H\circ T_{(m_2,m_1)}N.$$

\begin{proposition}
	Given two Nambu-Poisson relations $N: M_1\dashrightarrow M_2$ and $H: M_2\dashrightarrow M_3$ with clean composition $H\circ N: M_1 \dashrightarrow M_3$. Then $H\circ N$ is again a Nambu-Poisson relation.
\end{proposition}
\begin{proof}
	In order to show $H\circ N$ is a coisotropic submanifold. Let
	$$(\alpha_{3_j},-\alpha_{1_j})\in (T(H\circ N))^\circ,\hspace*{0.5em} 1\leq j \leq n-1$$
	be given, with base point $(m_3,m_1)\in H\circ N$. By clean composition, we can pick a point $m_2\in M_2$ with $(m_3,m_2)\in H$ and $(m_2,m_1)\in N$ such that
	$$T_{(m_3,m_1)}(H\circ N)=T_{(m_3,m_2)}H\circ T_{(m_2,m_1)}N.$$
	Thus, there are some differential 1-forms $\alpha_{2_j}\in T^*_{m_2}M_2,\hspace*{0.5em}1\leq j\leq n-1$ such that $(\alpha_{3_j},-\alpha_{2_j})\in (TH)^\circ$ and $(\alpha_{2_j},-\alpha_{1_j})\in (TN)^\circ$. Let $X_i=\pi^\sharp_i(\alpha_{i_1}\wedge \cdots \wedge \alpha_{i_{n-1}})$. Then we have $(X_3,X_2)\in TH$ and $(X_2,X_1)\in TN$, since $H,N$ are coisotropic. Thus, we obtain that $H\circ N$ is coisotropic since $(X_3,X_1)\in TH\circ TN=T(H\circ N)$.
\end{proof}

Note that $H\subseteq E$ is an $n$-Lie subalgebroid if and only if $\{\sigma\in \Gamma(E)|\sigma|_N\in \Gamma(H)\}$ is an $n$-Lie subalgebra, with $\{\sigma\in \Gamma(E)|\sigma|_N=0\}$ as an $n$-Lie ideal, that is  $\rho(\wedge^{n-1}H)\subseteq TN$. Therefore, for the dual picture, we have
\begin{eqnarray*}
	\sigma|_N\in \Gamma(H)\hspace*{0.5em}&\Leftrightarrow& \phi_\sigma \hspace*{0.5em}\mathrm{vanishes}\hspace*{0.5em}\mathrm{on} \hspace*{0.5em}H^\circ \subseteq E^*|_N,
	\\
	\sigma|_N=0\hspace*{0.5em}&\Leftrightarrow& \phi_\sigma \hspace*{0.5em}\mathrm{vanishes}\hspace*{0.5em}\mathrm{on} \hspace*{0.5em}E^*|_N.
\end{eqnarray*}
\begin{proposition}\label{pro2.7}
	Let  $E\to M$ be an $n$-Lie algebroid and $H\to N$ be a vector bundle. Then $H$ is an $n$-Lie subalgebroid if and only if $H^\circ \subseteq E^*$ is a coisotropic submanifold.
\end{proposition}
\begin{proof}
	Let $H^\circ \subseteq E$ be coisotropic. If $\sigma_1|_N,\cdots,\sigma_{n-1}|_N\in\Gamma(H)$ and $f|_N=0$, then $\phi_{\sigma_1},\cdots,\phi_{\sigma_{n-1}}$ and $p^*f$ vanish on $H^\circ$, so we have
	$$\{\phi_{\sigma_1},\cdots,\phi_{\sigma_{n-1}},p^*f\}=p^*(\rho(\sigma_1\wedge \cdots \sigma_{n-1})f).$$
	Since $H^\circ \subseteq E$ is a coisotropic submanifold, we have $\rho(\sigma_1\wedge \cdots \sigma_{n-1})f=0$,  which implies that $\rho(\sigma_1\wedge \cdots \wedge\sigma_{n-1})$ is tangent to $N$. Therefore, $\rho(\wedge^{n-1}H)\subseteq TN$ since $\sigma_1,\cdots,\sigma_{n-1}$ are any sections restricting to  sections of $N$. Similar, If $\sigma_1|_N,\cdots,\sigma_n|_N\in\Gamma(H)$, then $\phi_{\sigma_1}|_{H^\circ}=0,\cdots,\phi_{\sigma_n}|_{H^\circ}=0$, hence we have
	$$\{\phi_{\sigma_1},\cdots,\phi_{\sigma_n}\}=\phi_{[\sigma_1,\cdots,\sigma_n]}$$
	which implies that $[\sigma_1,\cdots,\sigma_n]|_N\subseteq \Gamma(H)$. This shows that $H$ is an $n$-Lie subalgebroid.
	
	Conversely, if $H$ is an $n$-Lie subalgebroid, then for all $\sigma_1|_N, \cdots,\sigma_n|_N \in \Gamma(H)$, and all $f_1|_N,\cdots,f_n|_N=0$, the Nambu-Poisson bracket
	\begin{eqnarray*}
		\{\phi_{\sigma_1},\cdots,\phi_{\sigma_n}\}&=&\phi_{[\sigma_1,\cdots,\sigma_n]},
		\\
		\{\phi_{\sigma_1},\cdots,\phi_{\sigma_{n-1},p^*f_i}\}&=&p^*(\rho(\sigma_1\wedge \cdots \sigma_{n-1})f_i),
		\\
		\{\phi_{\sigma_1},\cdots,\sigma_k,p^*f_{k+1},\cdots,p^*f_n\}&=&0,\hspace*{0.5em}\forall\hspace*{0.5em} 0\leq k \leq n-2
	\end{eqnarray*}
	all restrict to 0 on $H^\circ$. Since these functions generate the vanishing $n$-Lie ideal of $H^\circ$
	inside $C^\infty(M)$, then we have this ideal is an $n$-Lie subalgebra; that is, $H^\circ$ is coisotropic.
\end{proof}
\begin{remark}
	In Poisson geometry, there exist a nice symmetry:
	\begin{enumerate}
		\item[i)] For a Poisson manifold $(M,\pi)$, we have that $N\subseteq M$ is a coisotropic submanifold if and only if $(TN)^\circ \subseteq T^*$ is a Lie subalgebroid.
		\item[ii)] For a Lie algebroid $E$,  a vector subbundle $H\subseteq E$ is a Lie subalgebroid if and only if
		$(H)^\circ \subseteq E^*$
		is a coisotropic submanifold.
	\end{enumerate}
	However, for Nambu-Poisson manifold $(M,\pi)$, where $\pi\in \mathfrak{X}^l(M),\hspace*{0.5em}l\geq3,$ the first description is not true since $T^*M$ is not an $n$-Lie algebroid.
\end{remark}
\begin{definition}
	We denote by $\mathcal{VB^\vee}_{Nambu}$ the category of vector bundles with linear Nambu-Poisson
	structures of rank $n$, with comorphisms that are Nambu-Poisson relations.
\end{definition}
\begin{theorem}\label{theo6.10}
	Given two $n$-Lie algebroids $E_1 \to M_1$ and $E_2 \to M_2$ of rank $n$. Then $\Phi_{E}: E_1 \to E_2$ is an $n$-Lie algebroid morphism if and only if the dual comorphism $\Phi_{E^*}: E_1^* \dashrightarrow E_2^*$ is a Nambu-Poisson relation. We conclude that there is an equivalence of categories between
	$\mathcal{VB}^\vee_{Nambu}$ and $\mathcal{LA}.$ 
\end{theorem}
\begin{proof}
	By  the definition of $n$-Lie algebroid morphisms, we have that $\Phi_{E}$ is an $n$-Lie algebroid morphism if and only if its graph is an $n$-Lie subalgebroid. By Proposition \ref{pro2.7} and the definition of Nambu-Poisson relations, we conclude that  this is the case if and only if the dual comorphism $\Phi_{E^*}: E_1^* \dashrightarrow E_2^*$ is a Nambu-Poisson relation.
\end{proof}

\bibliographystyle{plain}

\begin{thebibliography}{50}	
	\bibitem{ref1}
	D. Alekseevsky and P. Guha,
	On decomposability of Nambu-Poisson tensor,
	\emph{Acta. Math. Univ. Comenian.} 65(1996), 1-9.
	
	\bibitem{ref2}
	J.~Arnlind, A. Makhlouf and S. Silvestrov,
	Construction of $n$-{Lie} algebras and $n$-ary {Hom}-{Nambu}-{Lie} algebras,
	\emph{J. Math. Phys.} 52(2011), 123502.	
	
	\bibitem{ref3}
	S.~Basu, S.~Basu, and A. Das, et al,
	Nambu structures and associated bialgebroids,
	\emph{Proc. Math. Sci.} 129(2019), 1-36.
	
	\bibitem{BC}
	Y.~Bi, Z.~C, and Z. C, et al,
	The geometric constraints on Filippov  algebroids,
	preprint, arXiv:2309.04180.	
	
	\bibitem{ref4}
	A. S. Cattaneo and M. Zambon,
	Coisotropic embeddings in {Poisson} manifolds,
	\emph{Trans. Amer. Math. Soc.} 361(7)(2009), 3721-3746.
	
	\bibitem{ref5}
	Z.~Chen and Z. J.~Liu,
	On (co-)morphisms of Lie pseudoalgebras and groupoids,
	\emph{J. Algebra} 316(1)(2007),  1-31.	
	
	\bibitem{ref6}
	M.~Crainic, R. L.~Fernande and L.~M\v{a}rcut,
	Lecture on Poisson {Geometry}, Graduate Studies in Mathematics Vol. 217,
	{Amer. Math. Soc.} 2021.	
	
	\bibitem{ref7}
	Y. L. Daletskii and L. A. Takhtajan,
	Leibniz and Lie algebra structures for {Nambu} algebra,
	\emph{Lett. Math. Phys.} 39(1997), 127-141.	
	\bibitem{ref8}
	A.~Das,
	Multiplicative {Nambu} structures on {Lie} groupoids,
	\emph{J. Ram. Math. Soc.} 35(3)(2020), 277-298
	\bibitem{ref9}
	V. T.~Filippov,
	$n$-Lie algebras,
	\emph{Siberian Math. J.} 26(1985), 879-891.	
	\bibitem{ref10}
	J.~Grabowski and G.~Marmo,
	On {Filippov} algebroids and multiplicative {Nambu}-{Poisson} structures,
	\emph{Differ. Geom. Appl.} 12(1)(2000), 35-50.	
	\bibitem{ref11}
	A. B.~Hassine, T.~Chtioui, M.~Elhamdadi and S. ~Mabrouk,
	Extensions and crossed modules of $n$-{Lie} {Rinehart} {algebras},
	\emph{Adv. Appl. Clifford Alg.} 32(2)(2022), 31.	
	\bibitem{ref12}
	L.~G. He, Z.~J. Liu and  D.~S. Zhong,
	Poisson actions and Lie bialgebroid morphisms,
	\emph{Contemp. Math.} 315(2002), 235-244.	
	\bibitem{ref13}
	R.~Hermann,
	\emph{Vector Bundles in Mathematical Physics}, New York: Addison Wesley Publishing Company, 1970.	
	\bibitem{ref14}
	P. J. Higgins and K.~Mackenzie,
	Algebraic constructions in the category of Lie algebroids,
	\emph{J. Algebra,} 129(1)(1990), 194-230.	
	\bibitem{ref15}
	P. J. Higgins and K.~Mackenzie,
	Duality for base-changing morphisms of vector bundles, modules, {Lie} algebroids and {Poisson} structures,
	\emph{Math. Proc. Cambridge Philos. Soc.} 114(3)(1993), 471-488.	
	\bibitem{ref16}
	J.~Huebschmann,
	Lie-Rinehart algebras, {Gerstenhaber} algebras and {Batalin}-{Vilkovisky} algebras,
	\emph{Annales de l'institut Fourier,} 48(2)(1998), 425-440.	
	\bibitem{ref17}
	J.~Huebschmann,
	Duality for {Lie}-{Rinehart} algebras and the modular class, \emph{J. reine angew. Math.} 510(1999), 103-159.	
	\bibitem{ref18}
	K.~Mackenzie,
	Lie algebroids and {Lie} pseudoalgebras,
	\emph{Bull. London Math. Soc.} 27(2)(1995), 97-147.	
	\bibitem{ref19}
	K.~Mackenzie,
	\emph{General theory of Lie groupoids and Lie algebroids}, Cambridge: Cambridge University Press, 2005.		
	\bibitem{ref20}
	G.~Marmo, G.~Vilasi and A. M.~Vinogradov,
	The local structure of $n$-{Poisson} and 	$n$-{Jacobi} manifolds,
	\emph{J. Geom. Phys.} 25(1-2)(1997), 141-182.
	\bibitem{ref21}
	E.~Meinrenken,
	Introduction to {Poisson} {Geometry}, lecture notes, 2017.	
	
	\bibitem{ref22}
	N.~Nakanishi,
	On {Nambu}-{Poisson} manifolds,
	\emph{Rev. Math. Phys.} 10(1998), 499-510.	
	\bibitem{ref23}
	Y. Sheng and C. Zhu,
	Higher extensions of {Lie} algebroids,
	\emph{Commun. Contemp. Math.} 19(3)(2017), 1650034.
	
	\bibitem{ref24}
	A.~Silva and A.~Weinstein,
	Geometric models for noncommutative algebras,
	Berkeley Mathematics Lecture Notes, Vol. 10.
	Amer. Math. Soc. 1999.
	
	\bibitem{ref25}
	L. Takhtajan,
	On foundation of the generalized {Nambu} mechanics,
	\emph{Commun. Math. Phys.} 160(1994), 295-315.	
	\bibitem{ref26}
	A.~Weinstein,
	Coisotropic calculus and {Poisson} groupoids,
	\emph{J. Math. Soc. Jpn.} 40(4)(1988), 705-727.			
\end{thebibliography}

\end{document}